\documentclass[a4paper,11pt,intlimits,oneside]{amsart}

\usepackage{amsfonts}
\usepackage{amsmath}   
\usepackage{amssymb}
\usepackage{amsthm}
\usepackage{mathtools}
\usepackage{tikz-cd}
\usepackage{xcolor}
\usepackage[colorlinks=true]{hyperref}

\newtheorem*{Th*}{Theorem}
\newtheorem{Th}{Theorem}[section]
\newtheorem{Prop}{Proposition}[section]   
\newtheorem{Lem}{Lemma}[section]   
\newtheorem{Coro}{Corollary}[section]   

\newtheorem{Rem}{Remark}[section]

\newcommand{\R}{\mathbb{R}}
\newcommand{\Z}{\mathbb{Z}}
\newcommand{\C}{\mathbb{C}}
\newcommand{\T}{\mathbb{T}}
\newcommand{\Q}{\mathcal{Q}}
\newcommand{\LL}{\mathcal{L}}

\newcommand{\hu}{{\widehat u}}
\newcommand{\h}{\mathfrak{h}}
\newcommand{\1}{\langle} 
\newcommand{\2}{\rangle} 
\newcommand{\n}{\langle n\rangle} 

\newcommand{\p}{{\partial}}
\newcommand{\HH}{\mathcal{H}}
\newcommand{\spec}{\mathop{\rm Spec}\nolimits}

\newcommand{\re}{\mathop{\rm Re}}

\newcommand{\image}{\mathop{\rm Image}}

\newcommand{\dom}{\mathop{\rm Dom}}
\newcommand{\e}{\varepsilon}
\begin{document}

\title[Low regularity phase space for BO equation]{On the low regularity phase space of the Benjamin-Ono equation}   
 
\dedicatory{\em Dedicated to the memory of our friend and collaborator Thomas Kappeler}
 
\author[P. G\'erard]{Patrick G\'erard}
\address{ Universit\'e Paris--Saclay, Laboratoire de Math\'ematiques d'Orsay, CNRS, UMR 8628, 91405 Orsay, 
France} \email{{\tt patrick.gerard@universite-paris-saclay.fr}}
\author[P. Topalov]{Petar Topalov}
\address{Department of Mathematics, Northeastern University,
567 LA (Lake Hall), Boston, MA 0215, USA}
\email{{\tt p.topalov@northeastern.edu}}
  
\begin{abstract}  
In this paper we prove that the Benjamin-Ono equation is globally in time $C^0$--well-posed in the Hilbert space 
$H^{-1/2,\sqrt{\log}}(\T,\R)$ of periodic distributions in $H^{-1/2}(\T,\R)$ with $\sqrt{\log}$-weights.
The space $H^{-1/2,\sqrt{\log}}(\T,\R)$ can thus be considered as a maximal low regularity phase space for 
the Benjamin-Ono equation corresponding to the scale $H^s(\T,\R)$, $s>-1/2$.
\end{abstract}   

\maketitle



\tableofcontents

\section{Introduction}\label{sec:ntroduction}
In this paper we study the Benjamin-Ono equation on the torus $\T := \R/2\pi\Z$,
\begin{equation}\label{eq:BO}
\partial_t u=\partial_x\big(|\partial_x|u-u^2\big), 
\end{equation}
where $u\equiv u(x,t)$, $x\in\T$, $t\in\R$, is real valued and $|\partial_x| : H^\beta_c\to H^{\beta-1}_c$, $\beta\in\R$, 
is the Fourier multiplier
\begin{equation}\label{eq:|D|}
|\partial_x| : \sum_{n\in\Z}\widehat{v}(n) e^{i n x}\mapsto\sum_{n\in\Z}|n|\,\widehat{v}(n) e^{i n x}
\end{equation}
where $\widehat{v}(n)$, $n\in\Z$, are the Fourier coefficients of $v\in H^\beta_c$ and $H^\beta_c\equiv H^\beta(\T,\C)$
is the Sobolev space of complex valued distributions on the torus $\T$. 
The equation \eqref{eq:BO} was introduced in 1967 by Benjamin \cite{Ben1967} and Davis $\&$ Acrivos \cite{DA1967} as 
a model for a special regime of internal gravity waves at the interface of two fluids. It is well known that \eqref{eq:BO} admits 
a Lax pair representation (cf. \cite{Na1979}) that leads to an infinite sequence of conserved quantities (cf. \cite{Na1979}, \cite{BK1979})
and that it can be written in Hamiltonian form  with Hamiltonian
\begin{equation}\label{eq:BO-Hamiltonian}
\HH(u):=\frac{1}{2\pi}\int_0^{2\pi}\Big(\frac{1}{2}\big(|\p_x|^{1/2}u\big)^2-\frac{1}{3}u^3\Big)\,dx
\end{equation}
by the use of the Gardner bracket
\begin{equation}\label{eq:bracket}
\{F,G\}(u):=\frac{1}{2\pi}\int_0^{2\pi}\big(\partial_x\nabla_u F\big)\nabla_u G\,dx
\end{equation}
where $\nabla_u F$ and $\nabla_u G$ are the $L^2$-gradients of $F,G\in C^1(H^\beta_c,\R)$ at $u\in H^\beta_c$.
By the Sobolev embedding $H^{1/2}_c\hookrightarrow L^3(\T,\C)$, the Hamiltonian \eqref{eq:BO-Hamiltonian} is 
well defined and analytic on $H^{1/2}_c$, the {\em energy space} of \eqref{eq:BO}.
The problem of the existence and the uniqueness of the solutions of the Benjamin-Ono equation is well studied 
-- see \cite{GK,GKT1}, \cite{Saut2019,KS} and references therein. 
We refer to \cite{Saut2019,KS} for an excellent survey and a derivation of \eqref{eq:BO}.

By using the Hamiltonian formalism for \eqref{eq:BO}, it was recently proven in \cite{GK},\cite{GKT1} that for 
any $s>-1/2$, the Benjamin-Ono equation has a {\em homeomorphic} Birkhoff map
\begin{equation}\label{eq:Phi-introduction}
\Phi : H^{s}_{r,0}\to\h^{\frac{1}{2}+s}_{r,0},\ u\mapsto
\big((\overline{\Phi_{-n}(u)}  )_{n \le -1}, (\Phi_n(u))_{n \ge 1},\Phi_0(u)=0\big) 
\end{equation}
where for $\beta\in\R$, 
\begin{equation}\label{qe:H_0}
H^\beta_{r,0}:=\big\{u\in H^\beta_r\,\big|\,\hu(0)=0, \overline{u}=u\big\},\,\,
\end{equation}
and
\begin{equation}\label{eq:h_{r,0}}
\h^\beta_{r,0}:=\big\{z\in\h^\beta_c\,\big|\,z_0=0, z_{-n}=\overline{z}_n \ \forall n\ge 1\big\}
\end{equation}
is a real subspace in the Hilbert space of complex-valued sequences
\begin{equation}\label{eq:h_{c,0}}
\h^\beta_c:=\big\{ (z_n)_{n\in\Z}\,\big|\,\sum_{n\in\Z}\n^{2\beta}|z_n|^2 < \infty\big\},\quad\n:=\max(1,|n|),
\end{equation}
equipped with the norm $\|z\|_{\h^\beta_c}:=\big(\sum_{n\in\Z}\n^{2\beta}|z_n|^2\big)^{1/2}$.
For $\beta=0$ we set 
\[
L^2_{r,0}\equiv H^0_{r,0},\quad L^2_c\equiv H^0_c,\quad\ell^2_{r,0}\equiv\h^0_{r,0},\quad\ell^2_c\equiv\h^0_c\,.
\]
By \cite{GKT3,GKT4}, the Birkhoff map \eqref{eq:Phi-introduction} is a {\em bianalytic} diffeomorphism.
It transforms the trajectories of the Benjamin-Ono equation \eqref{eq:BO} into straight lines that have constant frequencies
on any given isospectral set (infinite torus) of potentials of the corresponding Lax operator (see \eqref{eq:L} below). 
In this sense, the Birkhoff map can be considered as a non-linear Fourier transform that significantly simplifies 
the solutions of \eqref{eq:BO}. This fact allows us to prove that for any $-1/2<s<0$, 
\eqref{eq:BO} is globally $C^0$--well-posed on $H^{s}_{r,0}$ (\cite{GKT1}) improving in this way 
the previously known well-posedness results (see \cite{Mo,MoP}). 
Additional applications of the Birkhoff map include the proof of the almost periodicity of the solutions of 
the Benjamin-Ono equation and the orbital stability of the Benjamin-Ono traveling waves
(see \cite[Theorem 3 and Theorem 4]{GKT1} and \cite{GKT4}).

In order to formulate our results we define for any $\beta\in\R$ the Hilbert space of 
periodic distributions in $H^\beta_c$,
\begin{equation}\label{eq:log-spaces}
H^{\beta,\sqrt{\log}}_c\equiv H^{\beta,\sqrt{\log}}_c(\T,\C):=
\big\{u\in H^{\beta}_c\,\big|\,\sum_{n\in\Z}\n^{2\beta}\log(\n+1)\,|\hu(n)|^2<\infty\big\}
\end{equation}
as well as the spaces 
\[
H^{\beta,\sqrt{\log}}_{r,0}:=H^\beta_{r,0}\cap H^{\beta,\sqrt{\log}}_c,
\]
\[
\h^{\beta,\sqrt{\log}}_{r,0}:=\h^\beta_{r,0}\cap\h^{\beta,\sqrt{\log}}_c, 
\]
and 
\[
\h^{\beta,\sqrt{\log}}_c:=\big\{(z_n)_{n\in\Z}\in\h^\beta_c\,\big|\,\sum_{n\in\Z}\n^{2\beta}\log(\n+1)\,|z_n|^2<\infty\big\}.
\] 
Note that for any $s>-1/2$ we have the compact embedding
\[
H^s_{r,0}\subsetneqq H^{-1/2,\sqrt{\log}}_{r,0}.
\]
Our first result concerns the extension of the Birkhoff map \eqref{eq:Phi-introduction} from $H^s_{r,0}$ with $s>-1/2$ to
the space $H^{-1/2,\sqrt{\log}}_{r,0}$.

\begin{Th}\label{th:main}
The Birkhoff map \eqref{eq:Phi-introduction} extends to a homeomorphic map
$\Phi : H^{-1/2,\sqrt{\log}}_{r,0}\to\h^{0,\sqrt{\log}}_{r,0}$.
\end{Th}

As a consequence from this theorem we obtain the following

\begin{Coro}\label{coro:well-posedness}
The Benjamin-Ono equation \eqref{eq:BO} is globally in time $C^0$--well-posed in the phase space $H^{-1/2,\sqrt{\log}}_{r,0}$.
More specifically, for any $t\in\R$ and $s>-1/2$ the flow map $S^t : H^s_{r,0}\to H^s_{r,0}$ defined in
\cite[Theorem 1]{GKT1} extends to a continuous flow map $S^t : H^{-1/2,\sqrt{\log}}_{r,0}\to H^{-1/2,\sqrt{\log}}_{r,0}$. 
Furthermore, for any $T>0$ the associated solution map 
$S : H^{-1/2,\sqrt{\log}}_{r,0}\to C\big([-T,T],H^{-1/2,\sqrt{\log}}_{r,0}\big)$, 
$u_0\mapsto\big\{t\mapsto S^t u_0, t\in[-T,T]\big\}$, 
is continuous and the corresponding trajectories are almost periodic as functions from $\R $ to $H^{-1/2,\sqrt{\log}}_{r,0}$.
\end{Coro}

\begin{Rem}
Recently, Killip, Laurens and Vi\c san \cite{KLV} found a different proof of the wellposedness on $H^s_r$ for every $s>-1/2$, 
which can be generalized to the Benjamin--Ono equation on the real line. It would be interesting to know whether the methods of 
\cite{KLV} lead to a similar wellposedness result on $H^{-1/2,\sqrt{\log }}_r(\R )$.
\end{Rem} 

\begin{Rem}
The first author recently derived in \cite{G} an explicit formula for the solution of the Benjamin--Ono equation on the torus. 
It can be easily checked that this formula holds for every initial datum in $H^{-1/2,\sqrt{\log}}_{r,0}(\T)$.
It does not seem straightforward to get Corollary \ref{coro:well-posedness} from this formula only.
\end{Rem}

Next we come to some important limitations of the above extension, which are specific to the bottom regularity 
$H^{-1/2,\sqrt{\log}}_{r,0}$. First we start with the lack of weak continuity.

\begin{Prop}\label{prop:noweakcontinuity}
The map $\Phi $ is not weakly continuous from  $H^{-1/2,\sqrt{\log}}_{r,0}$ to  $\h^{0,\sqrt{\log}}_{r,0}$, 
and the flow map of the Benjamin--Ono equation is not weakly continuous from $H^{-1/2,\sqrt{\log}}_{r,0}$ to  
$H^{-1/2,\sqrt{\log}}_{r,0}$. In fact, there exists a sequence of smooth initial data converging weakly to $0$ in 
$H^{-1/2,\sqrt{\log}}_{r,0}$, and such that the sequence of corresponding solutions does not converge to $0$ in
$\mathcal D'(\T)$ on any time interval $[0,T]$ with $T>0$.
\end{Prop}

The proof of Proposition \ref{prop:noweakcontinuity} consists in revisiting the counterexample of \cite{GKT1} and in observing
that the  $H^{-1/2,\sqrt{\log}}_{r,0}$ regularity is critical in this construction.

\medskip

The second limitation concerns the smoothness of the Birkhoff map and is in sharp contrast with the results of \cite{GKT3,GKT4}.

\begin{Prop}\label{prop:no_analyticity}
The (bi-analytic) Birkhoff map \eqref{eq:Phi-introduction} cannot be extended to an analytic map
\begin{equation}\label{eq:Phi-log spaces}
\Phi : H^{-1/2,\sqrt{\log}}_{r,0}\to\h^{0,\sqrt{\log}}_{r,0}\,.
\end{equation}
\end{Prop}

\noindent In fact, we prove that \eqref{eq:Phi-introduction} cannot be extended to a $C^2$-map
$$\Phi : H^{-1/2,\sqrt{\log}}_{r,0}\to\h^{0,\sqrt{\log}}_{r,0}$$ in a neighborhood of the origin.

\medskip

\noindent{\em Notation.}
In addition to the spaces introduced above we will also use the Hardy space 
\begin{equation}\label{eq:H_+}
H^\beta_+:=\big\{f\in H^\beta_c\,\big|\,\widehat{f}(n)=0\,\,\forall\,n<0\big\},\quad\beta\in\R,
\end{equation}
as well as the spaces of complex-valued sequences
\begin{align*}
\h^\beta_+&:=\big\{(z_n)_{n\ge 1}\,\big|\,\sum_{n\ge 1}\n^{2\beta}|z_n|^2<\infty\big\},\\
\h^{\beta,\sqrt{\log}}_+&:=\big\{(z_n)_{n\ge 1}\in\h^\beta_+\,\big|\,
\sum_{n\ge 1}\n^{2\beta}\log(\n+1)\,|z_n|^2<\infty\big\},
\end{align*}
and 
\[
\h^\beta_{\ge 0}:=\big\{(z_n)_{n\ge 0}\,\big|\,\sum_{n\ge 0}\n^{2\beta}|z_n|^2<\infty\big\}\,.
\]
We will denote the norm in $H^\beta_c$ by $\|\cdot\|_\beta$ and set $\|\cdot\|:=\|\cdot\|_0$. Similarly,
the norm in $H^{\beta,\sqrt{\log}}_c$ will be denoted by $\|\cdot\|_{\beta,\sqrt{\log}}$, and
the norm in $\h^{\beta,\sqrt{\log}}_c$ (resp. $\h^\beta_{\ge 0}$)
will be denoted by $\|\cdot\|_{\h^{\beta,\sqrt{\log}}_c}$ (resp. $\|\cdot\|_{\h^\beta_{\ge 0}}$). 
For $\beta=0$ we set $L^2_+\equiv H^0_+$, $\ell^2_+:=\h^0_+$, 
and $\ell^2_{\ge 0}:=\h^0_{\ge 0}$. We will also need the Banach space $\ell^1_+$ of 
complex-valued absolutely summable sequences $(z_n)_{n\ge 1}$ and the quadratic forms
\[
\1 f|g\2:=\frac{1}{2\pi}\int_0^{2\pi}f(x)\overline{g(x)}\,dx,\quad
\1 f,g\2:=\frac{1}{2\pi}\int_0^{2\pi}f(x)g(x)\,dx,\quad f,g\in L^2_c.
\]
Now, take $u\in H^{s}_c$ with $s>-1/2$ and consider the pseudo-differential expression
\begin{equation}\label{eq:L}
L_u:=D-T_u
\end{equation}
where $D:=-i\partial_x$ and $T_u: H^{1+s}_+\to H^{s}_+$ is the {\em Toeplitz operator} 
\begin{equation*}
T_u f:=\Pi(u f),\quad f\in H^{1+s}_+,
\end{equation*}
where $\Pi\equiv\Pi^+ : H^{s}_c \to H^{s}_+$ is the {\em Szeg\H o projector}
\[
\Pi : H^{s}_c\to H^{s}_+,\quad \sum_{n\in\Z}\widehat v(n) e^{i n x}\mapsto\sum_{n\ge 0}\widehat v(n) e^{i n x},
\]
onto the Hardy space $H^{s}_+$, introduced in \eqref{eq:H_+}.
Note that when restricted to $H^{1+s}_+$, $D$ coincides with the Fourier multiplier \eqref{eq:|D|}.
An important role in the integrability of the Benjamun-Ono equation is played by the {\em shift operator} 
$S : H^\beta_+\to H^\beta_+$, $f(x)\mapsto e^{ix} f(x)$, $\beta\in\R$ (cf. \cite{GK}).
It is not hard to see (cf. e.g. \cite[Lemma 1 (ii)]{GKT2}) that for any given $u\in H^s_c$ with $s>-1/2$ the pseudo-differential 
expressions $T_u$ and $L_u$ define bounded linear maps
\begin{equation}\label{eq:T,L,s>-1/2}
T_u :  H^{1+s}_+\to H^{s}_+\quad\text{and}\quad L_u :  H^{1+s}_+\to H^{s}_+.
\end{equation}
By Corollary \ref{coro:Toeplitz_operator} below this does {\em not} extend to log-spaces with $s=-1/2$.

\section{The Lax operator in log-spaces}\label{sec:Lax_operator}
In this Section we establish the basic properties of the Lax operator \eqref{eq:L} with potential $u\in H^{-1/2,\sqrt{\log}}_c$.
In view of Corollary \ref{coro:main_inequality}, for $u\in H^{-1/2,\sqrt{\log}}_c$ the pseudo-differential expression $L_u\equiv D-T_u$ 
defines a continuous map
\begin{equation*}
L_u : H^{1/2}_+\to H^{-1/2}_+\,.
\end{equation*}
In what follows we will think of $L_u$ as an unbounded operator on $H^{-1/2}_+$ with domain $\dom(L_u)=H^{1/2}_+$.
Let us fix $u_0\in L^2_c$ and choose $u\in B_{-1/2,\sqrt{\log}}(u_0)$,
\begin{equation}\label{eq:u_0-local}
B_{-1/2,\sqrt{\log}}(u_0):=\big\{u\in H^{-1/2,\sqrt{\log}}_c\,\big|\,\|u-u_0\|_{-1/2,\sqrt{\log}}<1/(4K_0)\big\}
\end{equation}
where $K_0>0$ is the constant appearing in Corollary \ref{coro:main_inequality}.
As in \cite{GKT1} consider the sesquilinear form
\begin{equation}\label{eq:Q-form}
\Q_{u_0,\lambda}(f,g):=\1-i\partial_xf|g\2-\big\1\Pi(uf)\big|g\big\2-\lambda\1 f|g\2
\end{equation}
where 
\[
\lambda\in\Lambda_{u_0}:=\big\{\lambda\in\C\,\big|\re(\lambda)<-\big(1+\eta(\|u_0\|)\big)\big\},
\]
\[
\eta(\|u_0\|):=C_0\|u_0\|\big(1+\|u_0\|\big),
\]
and $C_0>0$ is a positive constant. One easily sees that, with $C_0>0$ appropriately chosen, 
\cite[Lemma 1]{GKT1} continues to hold for complex-valued $u\in H^s_c$, $-1/2<s\le 0$. Then, we set $s=0$ and argue as in
the proof of \cite[Lemma 2]{GKT1} to obtain the following lemma.

\begin{Lem}\label{lem:Q-estimate}
There exist a constant $K>0$ such that for any $u_0\in L^2_c$ and for any $f,g\in H^{1/2}_+$ one has
\begin{equation*}
\frac{1}{2}\,\|f\|_{1/2}^2\le\big|\big\1\Q_{u_0,\lambda}f\big|f\big\2\big|,\quad
\big|\big\1\Q_{u_0,\lambda}f\big|g\big\2\big|\le\,\big(1+|\lambda|+K\|u_0\|\big)\|f\|_{1/2}\|g\|_{1/2}\,.
\end{equation*}
uniformly in $\lambda\in\Lambda_{u_0}$. If $u_0\in L^2_{r,0}$ and $\lambda\in\Lambda_{u_0}\cap\R$ then
$\big\1\Q_{u_0,\lambda}f\big|f\big\2\ge 0$ for any $f\in H^{1/2}_+$.
\end{Lem}

Now, we apply the Lax-Milgram lemma to obtain from Lemma \ref{lem:Q-estimate} that
for any $\lambda\in\Lambda_{u_0}$ the continuous map
\[
L_{u_0}-\lambda : H^{1/2}_+\to H^{-1/2}_+
\]
is a linear isomorphism such that
\begin{equation}\label{eq:L_0-resolvent}
\big\|(L_{u_0}-\lambda)^{-1}\big\|_{H^{-1/2}_+\to H^{1/2}_+}\le 2\,.
\end{equation}
This implies that for any $u\in B_{-1/2,\sqrt{\log}}(u_0)$ and $\lambda\in\Lambda_{u_0}$ we have
\begin{align}\label{eq:L-decomposition}
L_u-\lambda=L_{u_0}-\lambda-T_{\widetilde{u}}=
\big(I-T_{\widetilde{u}}(L_{u_0}-\lambda)^{-1}\big)(L_{u_0}-\lambda)
\end{align}
where $\widetilde{u}:=u-u_0$ and $I$ is the identity.
It follows from Corollary \ref{coro:main_inequality} and the fact that $u\in B_{-1/2,\sqrt{\log}}(u_0)$ that
\begin{equation}\label{eq:T-estimate}
\|T_{\widetilde{u}}\|_{H^{1/2}_+\to H^{-1/2}_+}\le K_0\|\widetilde{u}\|_{-1/2,\sqrt{\log}}<1/4.
\end{equation}
By combining this with \eqref{eq:L_0-resolvent} we see that
$\|T_{\widetilde{u}}(L_{u_0}-\lambda)^{-1}\|_{H^{-1/2}_+\to H^{-1/2}_+}<1/2$
and hence, in view of \eqref{eq:L-decomposition}, the map $L_u-\lambda : H^{1/2}_+\to H^{-1/2}_+$ is 
a linear isomorphism such that
\begin{equation}\label{eq:Neumann_series}
(L_u-\lambda)^{-1}=(L_{u_0}-\lambda)^{-1}\sum_{k\ge 0}\big[T_{\widetilde{u}}(L_{u_0}-\lambda)^{-1}\big]^k
\end{equation}
where the Neumann series converges in $\LL\big(H^{-1/2}_+,H^{-1/2}_+\big)$ uniformly in $u\in B_{-1/2,\sqrt{\log}}(u_0)$ and 
$\lambda\in\Lambda_{u_0}$. In particular, the map 
\begin{equation*}
(L_u-\lambda)^{-1} : H^{-1/2}_+\to H^{1/2}_+
\end{equation*}
is bounded for $u$ and $\lambda$ chosen as above. As a consequence we obtain

\begin{Th}\label{th:L_u-complex}
For any given $u\in H^{-1/2,\sqrt{\log}}_c$ the pseudo-differential expression \eqref{eq:L} defines a closed operator
$L_u$ on $H^{-1/2}_+$ with domain $\dom(L_u)=H^{1/2}_+$. This operator has a compact resolvent, and hence a discrete
spectrum. Moreover, we have:
\begin{itemize}
\item[(i)] Take $u_0\in L^2_c$ and assume that $u\in B_{-1/2,\sqrt{\log}}(u_0)$. 
Then, the half-plane $\Lambda_{u_0}$ belongs to the resolvent set of $L_u$.
\item[(ii)] The map
\begin{equation}\label{eq:L-lambda}
(u,\lambda)\mapsto(L_u-\lambda)^{-1}, \quad B_{-1/2,\sqrt{\log}}(u_0)\times\Lambda_{u_0}
\to\LL\big(H^{-1/2}_+,H^{1/2}_+\big),
\end{equation}
is well-defined and analytic.
\end{itemize}
\end{Th}

\begin{proof}[Proof of Theorem \ref{th:L_u-complex}]
We already proved that for a given $u_0\in L^2_c$ and for any $u\in B_{-1/2,\sqrt{\log}}(u_0)$ and
$\lambda\in\Lambda_{u_0}$ the map \eqref{eq:L-lambda} is well-defined.
The analyticity of \eqref{eq:L-lambda} follows from \eqref{eq:L_0-resolvent} and 
the uniform convergence of the Neumann series in \eqref{eq:Neumann_series} in $\LL\big(H^{-1/2}_+,H^{-1/2}_+\big)$.
Since the embedding $H^{1/2}_+\subseteq H^{-1/2}_+$ is compact the map
$(L_u-\lambda)^{-1} : H^{-1/2}_+\to H^{1/2}_+\subseteq H^{-1/2}_+$ is compact for $\lambda\in\Lambda_{u_0}$.
This proves that for $u\in B_{-1/2,\sqrt{\log}}(u_0)$  the unbounded operator $L_u$ on $H^{-1/2}$ with domain 
$\dom(L_u)=H^{1/2}_+$ is closed and has a compact resolvent. Since the radius of the ball $B_{-1/2,\sqrt{\log}}(u_0)$ is
independent of the choice of $u_0\in L^2_c$ the above holds for any $u\in H^{-1/2,\sqrt{\log}}_c$.
\end{proof}

Let us now assume that the potential $u$ is {\em real-valued}, $u\in H^{-1/2,\sqrt{\log}}_{r,0}$, and set
$B^{r,0}_{-1/2,\sqrt{\log}}(u_0):=B_{-1/2,\sqrt{\log}}(u_0)\cap H^{-1/2,\sqrt{\log}}_{r,0}$.
We can then choose $u_0\in L^2_{r,0}$ such that $u\in B^{r,0}_{-1/2,\sqrt{\log}}(u_0)$ and define
the unbounded operator 
\begin{equation}\label{eq:L^+}
L_u^{\rm sym}:=L_u|_{\dom(L_u^{\rm sym} )}
\end{equation}
on $L^2_+$ with domain
\begin{equation}\label{eq:Dom_+}
\dom(L_u^{\rm sym} ):=(L_u-\lambda_\bullet)^{-1}(L^2_+)\subseteq H^{1/2}_+
\end{equation}
for some $\lambda_\bullet\in\Lambda_{u_0}\cap\R$. The map 
\begin{equation}\label{eq:C^+}
C_u^{\rm sym}:=(L_u^{\rm sym}-\lambda_\bullet)^{-1} : L^2_+\to L^2_+
\end{equation}
is a composition of the following bounded linear maps
\begin{equation}\label{eq:short_sequence}
\begin{tikzcd}
L^2_+\arrow[r, hook]&H^{-1/2}_+\arrow[r, near end, "(L_u-\lambda_\bullet)^{-1}"]&
\quad H^{1/2}_+\arrow[r, hook]&L^2_+\,.
\end{tikzcd}
\end{equation}
Hence $C_u^{\rm sym}$ is bounded and compact.
Since $u$ is real-valued, $\1 L_u f|g\2=\1 f|L_ug\2$ for any $f,g\in\dom(L_u^{\rm sym})$.
This implies that $C_u^{\rm sym}$ is symmetric, and hence self-adjoint. In particular, we obtain that $L_u^{\rm sym}$ is 
a self-adjoint operator in $L^2_+$ with domain $\dom(L_u^{\rm sym} )$ and compact resolvent.
Now, we can apply the Hilbert-Schmidt theorem to \eqref{eq:C^+} to obtain
the following specification of Theorem \ref{th:L_u-complex} in the case when 
$u\in H^{-1/2,\sqrt{\log}}_{r,0}$.

\begin{Th}\label{th:L_u-real}
Assume that $u\in H^{-1/2,\sqrt{\log}}_{r,0}$. Then, we have:
\begin{itemize}
\item[(i)] The operator $L_u^{\rm sym}$ defined by \eqref{eq:L^+} 
and  \eqref{eq:Dom_+} is self-adjoint on $L^2_+$ with domain $\dom(L_u^{\rm sym})$ dense 
in $H^{1/2}_+$ (and $L^2_+$). 
\item[(ii)] The operator $L_u^{\rm sym}$ has a compact resolvent and a discrete spectrum
\begin{equation}\label{eq:spec}
\spec(L_u^{\rm sym})=\big\{\lambda_0\le\lambda_1\le...\le\lambda_n\le\lambda_{n+1}\le...\big\}
\end{equation}
that consists of infinitely many simple (real) eigenvalues such that $\lambda_n\to\infty$ as $n\to\infty$ and
\begin{equation}\label{eq:lambda-inequality}
\lambda_{n+1}\ge 1+\lambda_n,\quad n\ge 0.
\end{equation}
The corresponding normalized eigenfunctions $f_n\in H^{1/2}_+$ ($n\ge 0$) form an orthonormal basis in $L^2_+$.
\item[(iii)] The operators $L_u^{\rm sym}$ and $L_u$ (cf. Theorem \ref{th:L_u-complex}) have the same eigenvalues 
and root spaces. In particular, the eigenvalues $\lambda_n$ ($n\ge 0$) of $L_u$ (when ordered as in \eqref{eq:spec})
are {\em simple} and depend {\em real analytically} on the potential $u\in H^{-1/2,\sqrt{\log}}_{r,0}$.
\item[(iv)] Take $u_0\in L^2_{r,0}$, $u\in B^{r,0}_{-1/2,\sqrt{\log}}(u_0)$, 
and $\lambda_\bullet\in\Lambda_{u_0}\cap\R$.
Then, there exist constants $0<c<C$ such that for any $f\in H^{1/2}_+$,
\begin{equation}\label{eq:Q-form(inequality)}
c\|f\|_{1/2}^2\le\big\1(L_u-\lambda_\bullet)f\big|f\big\2\le C(1+|\lambda_\bullet|)\|f\|_{1/2}^2\,.
\end{equation}
The constants in \eqref{eq:Q-form(inequality)} can be chosen uniform in $u\in B^{r,0}_{-1/2,\sqrt{\log}}(u_0)$ and 
$\lambda_\bullet\in\Lambda_{u_0}\cap\R$.
\end{itemize}
\end{Th}

\begin{proof}[Proof of Theorem \ref{th:L_u-real}]
Assume that $u\in H^{-1/2,\sqrt{\log}}_{r,0}$.
The fact that $L_u^{\rm sym}$ is selfadjoint is already proved. The density of the domain 
$\dom(L_u^{\rm sym})$ in $H^{1/2}_+$ follows from \eqref{eq:short_sequence} since $L^2_+$ is dense in 
$H^{-1/2}_+$ and $(L_u\textcolor{red}{-}\lambda_\bullet)^{-1} : H^{-1/2}_+\to H^{1/2}_+$ is a linear isomorphism. 
This proves item (i).

In order to prove item (ii), recall that $C_u^{\rm sym}\equiv(L_u^{\rm sym}-\lambda_\bullet)^{-1} : L^2_+\to L^2_+$ 
is compact and symmetric with respect to the scalar product $\1\cdot,\cdot\2$ on $L^2_+$.
Hence, we can apply the Hilbert-Schmidt theorem to conclude that there exists an orthonormal basis of eigenfunctions of $C_u^{\rm sym}$
in $L^2_+$. Since the kernel of $C_u^{\rm sym}$ is trivial, zero is not an eigenvalue of $C_u^{\rm sym}$. Hence there are
infinitely many real eigenvalues $\mu_n$ ($n\ge 0$) of $C_u^{\rm sym}$ that converge to zero in $\R$. 
In view of \eqref{eq:C^+} we then conclude that $\mu_n=1/(\lambda_n-\lambda_\bullet)$ where $\lambda_n$ ($n\ge 0$)
is the spectrum of $L_u^{\rm sym}$. The fact the spectrum of $L_u^{\rm sym}$ is bounded below follows from the fact that
any eigenfunction of $L_u^{\rm sym}$ is also an eigenfunction of $L_u$ with the same eigenvalue. This follows directly from
the definition \eqref{eq:L^+} and the inclusion $\dom(L_u^{\rm sym})\subseteq H^{1/2}_+\equiv\dom(L_u)$.
The simplicity of the spectrum of $L_u^{\rm sym}$ and the inequality \eqref{eq:lambda-inequality} can be obtained from
the max-min principle in the same way as in \cite[Proposition 2.2]{GK}.

Let us now prove item (iii). We already mentioned that any eigenfunction of $L_u^{\rm sym}$ is an eigenfunction of $L_u$ with 
the same eigenvalue. Let $\lambda\in\C$ be an eigenvalue of $L_u$ and let $V_\lambda\subseteq H^{1/2}_+$ be its 
(finite dimensional) root space. Since the root space is an invariant subspace of $L_u$ and since the operator 
$L_u|_{V_\lambda} : V_\lambda\to V_\lambda$ is symmetric with respect to the restriction of the scalar product 
$\1\cdot|\cdot\2$ to $V_\lambda$, we conclude that $\lambda$ is real and $V_\lambda$ consists of eigenvectors of $L_u$ with 
eigenvalue $\lambda$. The same argument shows that the eigenspaces $V_\lambda$ and $V_\mu$ of $L_u$ corresponding to 
different eigenvalues $\lambda\ne\mu$ are orthogonal in $L^2_+$. Hence, if $\lambda\in\R$ is an eigenvalue of $L_u$ that is not an 
eigenvalue of $L_u^{\rm sym}$ then its eigenfunction $f$ is orthogonal to the eigenfunctions $f_n$ ($n\ge 0$) of $L_u^{\rm sym}$, 
that contradicts the fact that $f_n$ ($n\ge 0$) is an orthonormal basis in $L^2_+$. Hence, $\lambda$ is an eigenvalue of $L_u^{\rm sym}$.
A similar argument also shows that the eigenspaces $V_\lambda$ of $L_u$ are one dimensional. This proves the first statement in $(iii)$. 
The analytic dependence of $\lambda_n\equiv\lambda_n(u)$ with $n\ge 0$ on $u\in H^{-1/2,\sqrt{\log}}_{r,0}$ then 
follows from Theorem \ref{th:L_u-complex} (ii), the simplicity of the eigenvalue $\lambda_n$, and the properties of 
the Riesz's projector.

(iv) Choose $u,u_0$ and $\lambda_\bullet$ as in item (iv).
As in \eqref{eq:L-decomposition} we have
\begin{equation}\label{eq:L,S-relation}
L_u-\lambda_\bullet=R(L_{u_0}-\lambda_\bullet)=(L_{u_0}-\lambda_\bullet)\widetilde{R}
\end{equation}
where 
\begin{equation}\label{eq:S}
R:=I-T_{\widetilde{u}}(L_{u_0}-\lambda_\bullet)^{-1}\quad\text{\rm and}\quad
\widetilde{R}:=I-(L_{u_0}-\lambda_\bullet)^{-1}T_{\widetilde{u}}
\end{equation}
and $\widetilde{u}\equiv u-u_0$. It follows from \eqref{eq:L_0-resolvent} and \eqref{eq:T-estimate} that
\[
\|T_{\widetilde{u}}(L_{u_0}-\lambda_\bullet)^{-1}\|_{H^{-1/2}_+\to H^{-1/2}_+}<1/2,\quad
\|(L_{u_0}-\lambda_\bullet)^{-1}T_{\widetilde{u}}\|_{H^{1/2}_+\to H^{1/2}_+}<1/2,
\]
uniformly on the choice of $u\in B_{-1/2,\sqrt{\log}}$ and $\lambda_\bullet\in\Lambda_{u_0}\cap\R$.
This implies that the operators $R : H^{-1/2}_+\to H^{-1/2}_+$ and
$\widetilde{R} : H^{1/2}_+\to H^{1/2}_+$ are linear isomorphisms
that have well-defined (as convergent power series) square roots
\[
\sqrt{R} : H^{-1/2}_+\to H^{-1/2}_+\quad\text{\rm and}\quad
\sqrt{\widetilde{R}} : H^{1/2}_+\to H^{1/2}_+
\]
that are also linear isomorphisms. Since the potentials $u,u_0$ and  the constant $\lambda_\bullet$
are real we obtain from \eqref{eq:S} that $\1 R f|g\2=\1 f|\widetilde{R}g\2$ for any $f\in H^{-1/2}_+$ and $g\in H^{1/2}_+$.
This, together with the second equality in \eqref{eq:L,S-relation} and the definition of the square roots as convergent power series
implies that
\[
\big\1\sqrt{R}f\big|g\big\2=\big\1 f\big|\sqrt{\widetilde{R}}g\big\2\quad\text{\rm and}\quad
\sqrt{R}(L_{u_0}-\lambda_\bullet)=(L_{u_0}-\lambda_\bullet)\sqrt{\widetilde{R}}\,.
\]
Then, for any $f\in H^{1/2}_+$,
\begin{align}
\big\1(L_u-\lambda_\bullet)f\big|f\big\2&=\big\1 R(L_{u_0}-\lambda_\bullet)f\big|f\big\2=
\big\1\sqrt{R}(L_{u_0}-\lambda_\bullet)f\big|\sqrt{\widetilde{R}}f\big\2\nonumber\\
&=\big\1(L_{u_0}-\lambda_\bullet)\sqrt{\widetilde{R}}f\big|\sqrt{\widetilde{R}}f\big\2.\label{eq:1/2-form}
\end{align}
Item (iv) now follows from \eqref{eq:1/2-form} and Lemma \ref{lem:Q-estimate}.
\end{proof}

Let us now take $u\in H^{-1/2,\sqrt{\log}}_{r,0}$ and choose $u_0\in L^2_{r,0}$ and
$\lambda_\bullet\in\Lambda_{u_0}$ as in Theorem \ref{th:L_u-real} (iv).
Following \cite{GK}, we consider the {\em $n$-th gap}
\[
\gamma_n(u):=\lambda_n(u)-\lambda_{n-1}(u)-1\ge 0,\quad n\ge 1,
\]
that is well-defined by Theorem \ref{th:L_u-real} (ii) and non-negative in view of \eqref{eq:lambda-inequality}. 
We have
\begin{equation}\label{eq:gamma-relation}
0\le\sum_{k=1}^n\gamma_k(u)=\lambda_n(u)-\lambda_0(u)-n, 
\end{equation}
and hence
\begin{equation}\label{eq:lambda-estimate1}
\lambda_n(u)\ge n+\lambda_0(u),\quad n\ge 1.
\end{equation}
It follows from the estimate \eqref{eq:Q-form(inequality)} that $\lambda_0(u)>\lambda_\bullet$ which implies
\begin{equation}\label{eq:lambda-estimate2}
\lambda_0(u)\ge-\big(1+\eta(\|u_0\|)\big)
\end{equation} 
in view of the arbitrariness of the choice of $\lambda_\bullet\in\Lambda_{u_0}$.
Hence, for a given $\lambda_\bullet\in\Lambda_{u_0}\cap\R$ we have that
\begin{equation}\label{eq:lambda-estimate3}
\lambda_0(u)-\lambda_\bullet\ge-\Big(\lambda_\bullet+\big(1+\eta(\|u_0\|)\big)\Big)>0
\end{equation} 
uniformly in $u\in B^{r,0}_{-1/2,\sqrt{\log}}(u_0)$.
Theorem \ref{th:L_u-complex} allows us to define the meromorphic function (cf. \cite{GK})
\[
\mathcal{H}_\lambda(u):= \big\1(L_u+\lambda)^{-1}1\,\big|\,1\big\2
\]
with poles at $\{-\lambda_n(u)\,|\,n\ge 0\}$.
By arguing as in the proof of \cite[Proposition 3.1]{GK} one sees that 
\[
\mathcal{H}_\lambda(u)=
\frac{1}{\lambda_0+\lambda}\prod_{n=1}^\infty\Big(1-\frac{\gamma_n}{\lambda_n+\lambda}\Big)
\]
where the infinite product converges absolutely and we set $\gamma_n\equiv\gamma_n(u)$, $n\ge 1$, 
and $\lambda_n\equiv\lambda_n(u)$, $n\ge 0$, for simplicity of notation. The arguments in the proof 
of \cite[Proposition 3.1]{GK} also show that one has the trace formula
\begin{equation}\label{eq:trace_formula}
\sum_{n=1}^\infty\gamma_n(u)=-\lambda_0(u)\ge 0
\end{equation}
where the sequence $\big(\gamma_n(u)\big)_{n\ge 1}$ is absolutely summable.
As a consequence from \eqref{eq:gamma-relation} and \eqref{eq:trace_formula} we obtain that
$\lambda_n-n=-\sum_{k>n}\gamma_k$ and hence $\lambda_n(u)\le n$.
By combining this with \eqref{eq:lambda-estimate1} and  \eqref{eq:lambda-estimate2} we then conclude that
\begin{equation}\label{eq:lambda-estimate}
n-\big(1+\eta(\|u_0\|)\big)\le n+\lambda_0(u)\le\lambda_n(u)\le n,\quad n\ge 0,
\end{equation}
uniformly in $u\in B^{r,0}_{-1/2,\sqrt{\log}}(u_0)$. Note that the estimate
\begin{equation}\label{eq:lambda-estimate*}
n+\lambda_0(u)\le\lambda_n(u)\le n,\quad n\ge 0,
\end{equation}
holds for any $u\in H^{-1/2,\sqrt{\log}}_{r,0}$.
Further, note that the statements of Lemma 2.5 and Lemma 2.7 in \cite{GK} are purely algebraic in nature 
and continue to hold for $u\in H^{-1/2,\sqrt{\log}}_{r,0}$ since, by Theorem \ref{th:L_u-real}, 
all the quantities involved are well-defined. In particular, we obtain that
\begin{equation*}
\1 f_0(u)|1\2\ne 0\quad\text{\rm and}\quad\1 f_n(u)|Sf_{n-1}(u)\2\ne 0,\quad n\ge 1,
\end{equation*}
where $f_n(u)$ ($n\ge 0$) is an orthonormal basis of eigenfunctions of $L_u$ in $L^2_+$ (see Theorem \ref{th:L_u-real} (ii)).
This allows us to choose the orthonormal basis $f_n(u)$ ($n\ge 0$) in a unique way by imposing the conditions
(cf. \cite[Definition 2.8]{GK})
\begin{equation}\label{eq:normalizing_conditions}
\1 f_0(u)|1\2>0\quad\text{\rm and}\quad\1 f_n(u)|Sf_{n-1}(u)\2>0,\quad n\ge 1.
\end{equation}
In what follows we will assume that $f_n(u)$, $n\ge 0$, denotes this particular orthonormal basis.
Note that Theorem \ref{th:L_u-complex} (ii), the simplicity of the eigenvalues of $L_u$, and 
the properties of the Riesz's projector imply that for any given $n\ge 0$ the map
\begin{equation}\label{eq:f_n-analyticity}
f_n : H^{-1/2,\sqrt{\log}}_{r,0}\to H^{-1/2}_+,\quad u\mapsto f_n(u),
\end{equation}
is real-analytic.\footnote{Here we ignore the complex structure on $H^{-1/2}_+$ and consider the space as real.}

As above, we fix $u_0\in L^2_{r,0}$ and choose  $\lambda_\bullet\in\Lambda_{u_0}\cap\R$. 
By Theorem \ref{th:L_u-real} (iv) there exist constants $0<c<C$ such that inequality \eqref{eq:Q-form(inequality)} 
holds uniformly in $u\in B^{r,0}_{-1/2,\sqrt{\log}}(u_0)$. 
In particular, this implies that the sesquilinear form
\begin{equation}\label{eq:Q-form(general)}
\Q_{u,\lambda_\bullet} : H^{1/2}_+\times H^{1/2}_+\to\C,\quad (f,g)\mapsto\big\1(L_u-\lambda_\bullet)f\big|g\big\2,
\end{equation}
gives an equivalent Hilbert structure in $H^{1/2}_+$. Since the system of eigenfunction
\begin{equation}\label{eq:basis}
\widetilde{f}_n:=f_n/\sqrt{\lambda_n-\lambda_\bullet},\quad n\ge 0,
\end{equation} 
is complete in $H^{1/2}_+$ and orthonormal with respect to \eqref{eq:Q-form(general)} we conclude that 
it is a basis in $H^{1/2}_+$. (Recall from \eqref{eq:lambda-estimate3} that $\lambda_n>\lambda_\bullet$.)
By the Parseval's  identity, for any $f\in H^{1/2}_+$,
\[
\Q_{u,\lambda_\bullet}(f,f)=\sum_{n=0}^\infty\big|\Q_{u,\lambda_\bullet}(f,\widetilde{f}_n)\big|^2
=\sum_{n=0}^\infty(\lambda_n-\lambda_\bullet)\big|\1 f|f_n\2|^2\,.
\]
This, together with \eqref{eq:Q-form(inequality)} implies that
\[
c\,\|f\|_{1/2}^2\le\sum_{n=0}^\infty(\lambda_n-\lambda_\bullet)\,\big|\1f|f_n\2\big|^2\le
C(1+|\lambda_\bullet|)\|f\|_{1/2}^2.
\]
By combining this with \eqref{eq:lambda-estimate} we then see that there exist constants $0<\varkappa_1<1$ 
independent of the choice of $u\in B^{r,0}_{-1/2,\sqrt{\log}}(u_0)$ such that for any $f\in H^{1/2}_+$,
\begin{equation}\label{eq:uniform_estimate_norms1}
\varkappa_1^2\,\sum_{n=0}^\infty\big(\n+1\big)\,|\widehat{f}(n)|^2\le\sum_{n=0}^\infty\big(\n+1\big)\,\big|\1f|f_n\2\big|^2
\le\frac{1}{\varkappa_1^2}\,\sum_{n=0}^\infty\big(\n+1\big)\,|\widehat{f}(n)|^2.
\end{equation}
Let us now consider the Fourier transform corresponding to the orthonormal basis $f_n$ ($n\ge 0$), 
\begin{equation}\label{eq:K_0}
K_{u;0} : L^2_+\to\ell^2_{\ge 0},\quad f\mapsto\big(\1 f|f_n\2\big)_{n\ge 0},
\end{equation}
together with its restriction to $H^{1/2}_+$,
\begin{equation}\label{eq:K_1/2}
K_{u;1/2}:=K_{u;0}|_{H^{1/2}_+} : H^{1/2}_+\to\h^{1/2}_{\ge 0}.
\end{equation}
Note that the image of \eqref{eq:K_1/2} is dense in $\h^{1/2}_{\ge 0}$ since it contains all finite sequences 
$c_n\in\C$ ($0\le n\le N$) with $N\ge 0$.
This, together with \eqref{eq:uniform_estimate_norms1} implies that \eqref{eq:K_1/2} is a linear isomorphism. 

\begin{Rem}\label{rem:basis}
The argument above also show that for any $u\in H^{-1/2,\sqrt{\log}}_{r,0}$ the orthonormal basis 
$f_n\in H^{1/2}_+$ ($n\ge 0$) in $L^2$ gives a basis in $H^{1/2}_+$ such that for any $f\in H^{1/2}_+$
the Fourier series $f=\sum_{n\ge 0}\1 f|f_n\2 f_n$ converges in $H^{1/2}_+$ and 
$\big(\1 f|f_n\2\big)_{n\ge 0}\in\h^{1/2}_{\ge 0}$. Since \eqref{eq:K_1/2} is an isomorphism we also see that
for any given $(x_n)_{n\ge 0}\in\h^{1/2}_{\ge 0}$ there exists $f\in H^{1/2}_+$ such that $\1f|f_n\2=x_n$, $n\ge 0$.
\end{Rem}

We then interpolate between the maps \eqref{eq:K_0} and \eqref{eq:K_1/2} as well as their inverses 
$K_{u;0}^{-1} : \ell^2_{\ge 0}\to L^2_+$ and $K_{u;1/2}^{-1} :  \h^{1/2}_{\ge 0}\to H^{1/2}_+$
to conclude (cf. \cite[Example 3, Appendix to \S\,IX.4]{ReedSimon}) that for any $u\in B^{r,0}_{-1/2,\sqrt{\log}}(u_0)$,
$0\le\theta\le 1$, and for any $f\in H^{1/2}_+$,
\begin{equation*}
\varkappa_1^2\sum_{n=0}^\infty\big(\n+1\big)^\theta\,|\widehat{f}(n)|^2\le
\sum_{n=0}^\infty\big(\n+1\big)^\theta\,\big|\1f|f_n\2\big|^2\le 
\frac{1}{\varkappa_1^2}\sum_{n=0}^\infty\big(\n+1\big)^\theta\,|\widehat{f}(n)|^2.
\end{equation*}
By integrating these inequalities with respect to $\theta$ on the interval $0\le\theta\le 1$ we obtain that
for any $f\in H^{1/2}_+$,
\begin{equation}\label{eq:uniform_estimate_norms2}
\varkappa_1^2\,\|f\|_{1/2,1/\sqrt{\log}}^2\le
\sum_{n=0}^\infty\frac{\n}{\log(\n+1)}\,\big|\1f|f_n\2\big|^2\le 
\frac{1}{\varkappa_1^2}\,\|f\|_{1/2,1/\sqrt{\log}}^2
\end{equation}
with $0<\varkappa_1<1$ independent of the choice of $u\in B^{r,0}_{-1/2,\sqrt{\log}}(u_0)$.
This inequality implies that the map
\begin{equation}\label{eq:K_1/2-}
K_{u;1/2,1/\sqrt{\log}} : H^{1/2,1/\sqrt{\log}}_+\to\h^{1/2,1/\sqrt{\log}}_{\ge 0},\quad f\mapsto\big(\1 f|f_n\2\big)_{n\ge 0},
\end{equation}
is a linear isomorphism. Hence, the map conjugate to \eqref{eq:K_1/2-} with respect to the $L^2_+$- and the 
$\ell^2_{\ge}$-pairing,
\[
K_{u;1/2,1/\sqrt{\log}}^* : \h^{-1/2,\sqrt{\log}}_{\ge 0}\to H^{-1/2,\sqrt{\log}}_+,
\]
and its inverse
\begin{subequations}
\begin{equation}\label{eq:K_-1/2+}
K_{u;-1/2,\sqrt{\log}}:=\big(K_{u;1/2,1/\sqrt{\log}}^*\big)^{-1} : 
H^{-1/2,\sqrt{\log}}_+\to\h^{-1/2,\sqrt{\log}}_{\ge 0}
\end{equation}
are also linear isomorphisms. It is a straightforward task to see that
\begin{equation}\label{eq:K_-1/2+(bis)}
K_{u;-1/2,\sqrt{\log}}\,f=\big(\1f|f_n\2\big)_{n\ge 0},\quad\forall f\in H^{-1/2,\sqrt{\log}}_+,
\end{equation}
\end{subequations}
and that, in view of \eqref{eq:uniform_estimate_norms2},
\begin{equation}\label{eq:uniform_estimate_norms3}
\varkappa_1\|f\|_{-1/2,\sqrt{\log}}\le
\big\|K_{u;-1/2,\sqrt{\log}}\,f\big\|_{\h^{-1/2,\sqrt{\log}}_{\ge 0}}
\le\frac{1}{\varkappa_1}\,\|f\|_{-1/2,\sqrt{\log}}
\end{equation}
uniformly in $f\in H^{-1/2,\sqrt{\log}}_+$ and $u\in B^{r,0}_{-1/2,\sqrt{\log}}(u_0)$. 
Let
\[
\widetilde{\dom}(L_u):=(L_u-\lambda_\bullet)^{-1}\big(H^{-1/2,\sqrt{\log}}_+\big)\subseteq H^{1/2}_+
\]
be the domain of the operator $L_u$ in $H^{-1/2,\sqrt{\log}}_+$.
For any $u\in H^{-1/2,\sqrt{\log}}_{r,0}$ and $\lambda\in\R$ we have the following commutative diagram
\begin{equation}\label{eq:diagram1}
\begin{tikzcd}
H^{1/2}_+\arrow[d,"L_u-\lambda"]&\widetilde{\dom}(L_u)\arrow[l,hook]
\arrow[d,"L_u-\lambda"]\arrow[r,"K_u^{(1)}"]&
\h^{1/2,\sqrt{\log}}_{\ge 0}\arrow[d,"\ell_u-\lambda"]\\
H^{-1/2}_+&H^{-1/2,\sqrt{\log}}_+\arrow[l,hook]
\arrow[ur,"D_{u,\lambda}" near end]\arrow[r,"K_u^{(2)}"]&\h^{-1/2,\sqrt{\log}}_{\ge 0}
\end{tikzcd}
\end{equation}
where $\ell_u : \h^{1/2,\sqrt{\log}}_{\ge 0}\to\h^{-1/2,\sqrt{\log}}_{\ge 0}$ is the multiplication
$(z_n)_{n\ge 0}\mapsto\big(\lambda_n z_n\big)_{n\ge 0}$, and the maps $K_u^{(1)}$ and $K_u^{(2)}$ stand 
for $K_{u;1/2}|_{\widetilde{\dom}(L_u)}$ and $K_{u;-1/2,\sqrt{\log}}$ (see \eqref{eq:K_-1/2+},\eqref{eq:K_-1/2+(bis)}). 
For $\lambda$ in the resolvent set of $L_u$ there is a well-defined ``diagonal'' map (cf. \eqref{eq:lambda-estimate})
\begin{equation}\label{eq:D-operator}
D_{u,\lambda}:=(\ell_u-\lambda)^{-1}K_{u;-1/2,\sqrt{\log}}\,.
\end{equation}
The operator  $D_{u;\lambda}$ will play an important role in our study of the Birkhoff map. 
The map $K_{u;-1/2,\sqrt{\log}}$ as well the maps $\ell_u-\lambda$, and $D_{u;\lambda}$
(for $\lambda$ in the resolvent set of $L_u$) are linear isomorphisms.

\begin{Rem}\label{rem:diagram1}
By Corollary \ref{coro:Toeplitz_operator} the space $\widetilde{\dom}(L_u)$ does {\em not} coincide with 
$H^{1/2,\sqrt{\log}}_+$. This is in contrast with the case of the scale $H^s_{r,0}$, $s>-1/2$ 
(cf. \cite[Lemma 6]{GKT1}).
\end{Rem}

We will need an (improved) explicit formula for the constant appearing on the left side of \eqref{eq:uniform_estimate_norms3}.
To this end, we fix $u\in H^{-1/2,\sqrt{\log}}_{r,0}$ and consider the operator $L_u$ on $H^{-1/2}_+$ 
with domain $H^{1/2}_+$ (cf. Theorem \ref{th:L_u-complex}). It follows from Remark \ref{rem:basis} that
for any $f\in H^{1/2}_+$,
\[
\big\1(L_u-\lambda_0)f\big|f\big\2=\sum_{n\ge 0}(\lambda_n-\lambda_0)|\1f|f_n\2|^2\ge 0,
\]
where $\lambda_0\equiv\lambda_0(u)$ is the first eigenvalue of $L_u$. 
This together with Corollary \ref{coro:main_inequality} then implies that
\begin{align*}
\|f\|^2&\le\big\1(L_u-\lambda_0+1)f\big|f\big\2=\1-i\partial_x f|f\2-\big\1\Pi(uf)\big|f\big\2+(-\lambda_0+1)\|f\|^2\nonumber\\
&\le\big(2-\lambda_0+K_0\|u\|_{-1/2,\sqrt{\log}}\big)\|f\|_{1/2}^2,\quad f\in H^{1/2}_+\,.
\end{align*}
Hence, for any $f\in H^{1/2}_+$,
\[
\|f\|^2\le\big\1(L_u-\lambda_0+1)f\big|f\big\2\le M_u\|f\|_{1/2}^2,
\]
where 
\begin{equation}\label{eq:M_u}
M_u:=2-\lambda_0+K_0\|u\|_{-1/2,\sqrt{\log}}.
\end{equation}
This inequality together with \eqref{eq:lambda-estimate*} implies that
there exists a constant $C_{\lambda_0}\ge 1$ (chosen uniformly on bounded sets of $\lambda_0$ in $\R_{\le 0}$)
such that for any $f\in H^{1/2}_+$,
\begin{equation}\label{eq:M_u-inequality}
\|K_{u;1/2}f\|_{h^{1/2}_{\ge 0}}^2\le C_{\lambda_0} M_u\|f\|_{1/2}^2,
\end{equation}
where $K_{u;1/2} : H^{1/2}_+\to\h^{1/2}_{\ge 0}$ is the linear isomorphism \eqref{eq:K_1/2}.
As above, we then interpolate between the maps $K_{u;1/2} : H^{1/2}_+\to\h^{1/2}_{\ge 0}$ and 
$K_{u;0} : L^2_+\to\ell^2_{\ge 0}$ to obtain that for any $0\le\theta\le 1$ and $f\in H^{1/2}_+$,
\[
\sum_{n\ge 0}\big(\n+1\big)^\theta|\1f|f_n\2|^2\le 
\big(C_{\lambda_0} M_u\big)^\theta\sum_{n\ge 1}\big(\n+1\big)^\theta|\widehat{f}(n)|^2
\]
where, in order to accommodate the slight change of the weights, we choose $C_{\lambda_0}\ge 1$ larger if necessary. 
By integrating this inequality with respect to $\theta$ on the interval $0\le\theta\le 1$ we conclude that 
for any $f\in H^{1/2}_+$,
\[
\sum_{n\ge 0}\frac{\n}{\log(\n+1)}\,|\1f|f_n\2|^2\le 
C_{\lambda_0} M_u\sum_{n\ge 1}\frac{\n}{\log(\n+1)}\,|\widehat{f}(n)|^2.
\]
This implies that for any $f\in H^{1/2,1/\sqrt{\log}}_+$,
\[
\big\|K_{u;1/2,1/\sqrt{\log}} f\big\|_{\h^{1/2,1/\sqrt{\log}}_{\ge 0}}^2\le C_{\lambda_0} M_u\|f\|_{1/2,1/\sqrt{\log}}^2\,.
\]
We then argue by duality (as in the proof of \eqref{eq:uniform_estimate_norms3}) to conclude that
\begin{equation}\label{eq:uniform_estimate_norms3*}
\|f\|_{-1/2,\sqrt{\log}}^2\le C_{\lambda_0} M_u\big\|K_{u;-1/2,\sqrt{\log}}\,f\big\|_{\h^{-1/2,\sqrt{\log}}_{\ge 0}}^2
\end{equation}
where $K_{u;-1/2,\sqrt{\log}} : H^{-1/2,\sqrt{\log}}_{r,0}\to\h^{-1/2,\sqrt{\log}}_{\ge 0}$ is the linear
isomorphism \eqref{eq:K_-1/2+}, \eqref{eq:K_-1/2+(bis)}.
Further, we set $f\equiv\Pi u$ in \eqref{eq:uniform_estimate_norms3*} and use that
$\1\Pi u|f_n\2=-\1 L_u 1|f_n\2=-\lambda_n\1 1|f_n\2$, $n\ge 0$, to conclude from \eqref{eq:lambda-estimate*} that
\begin{equation}\label{eq:uniform_estimate_norms4*}
\frac{1}{2}\,\|u\|_{-1/2,\sqrt{\log}}^2\le 
C_{\lambda_0} M_u\big\|\big(\1 1|f_n(u)\2\big)_{n\ge 0}\big\|_{\h^{1/2,\sqrt{\log}}_{\ge 0}}^2
\end{equation}
with (possibly different) $C_{\lambda_0}\ge 1$ chosen uniformly on bounded sets of $\lambda_0$ in $\R_{\le 0}$.
In view of \eqref{eq:M_u} the estimate \eqref{eq:uniform_estimate_norms4*} can be written in the form
\[
-\frac{1}{2}\,\|u\|_{-1/2,\sqrt{\log}}^2+A\|u\|_{-1/2,\sqrt{\log}}+B\ge 0
\]
where 
\[
A:=K_0C_{\lambda_0}\big\|\big(\1 1|f_n\2\big)_{n\ge 0}\big\|_{\h^{1/2,\sqrt{\log}}_{\ge 0}}^2,\quad
B:=(2-\lambda_0) C_{\lambda_0}\big\|\big(\1 1|f_n\2\big)_{n\ge 0}\big\|_{\h^{1/2,\sqrt{\log}}_{\ge 0}}^2.
\]
This implies that $\|u\|_{-1/2,\sqrt{\log}}$ is bounded by the two roots of the quadratic polynomial 
$-\frac{1}{2}z^2+Az+B=0$. Hence, for any given $R>0$ we can choose the constant $C_{\lambda_0}\ge 1$
and a constant $C_R>0$ such that for any $u\in H^{-1/2,\sqrt{\log}}_{r,0}$ that satisfies
\begin{subequations}
\begin{equation}\label{eq:pre-Birkhoff_inverse_estimate2}
\max\Big(|\lambda_0(u)|, \big\|\big(\1 1|f_n(u)\2\big)_{n\ge 0}\big\|_{\h^{1/2,\sqrt{\log}}_{\ge 0}}\Big)\le R
\end{equation}
we have that $C_{\lambda_0}M_u\le C_R$, and hence, by \eqref{eq:D-operator} and \eqref{eq:uniform_estimate_norms3*},
\begin{equation}\label{eq:pre-Birkhoff_inverse_estimate1}
\big\|(D_{u,\lambda_0(u)-1})^{-1}\big\|_{\h^{1/2,\sqrt{\log}}_{\ge 0}\to H^{-1/2,\sqrt{\log}}_+}<C_R\,.
\end{equation}
\end{subequations}
Summarizing the above, we obtain the main result in this section.

\begin{Prop}\label{prop:pre-Birkhoff_estimates}
For any $u\in H^{-1/2,\sqrt{\log}}_{r,0}$ we have that
$\big(\1 1|f_n(u)\2\big)_{n\ge 0}\in\h^{1/2,\sqrt{\log}}_{\ge 0}$ and
\begin{equation}\label{eq:main_identity}
D_{u,\lambda_0(u)-1}\big(-\Pi u-\lambda_0(u)+1\big)=\big(\1 1|f_n(u)\big)_{n\ge 0}
\end{equation}
where $D_{u,\lambda} : H^{-1/2,\sqrt{\log}}_+\to\h^{1/2,\sqrt{\log}}_{\ge 0}$ is given by \eqref{eq:D-operator}.
Moreover, one has:
\begin{itemize}
\item[(i)] For any $v\in H^{-1/2,\sqrt{\log}}_{r,0}$ there exist constant $C\equiv C_v>0$  and an open neighborhood 
$U(v)$ of $v$ in $H^{-1/2,\sqrt{\log}}_{r,0}$ such that 
\begin{equation}\label{eq:pre-Birkhoff_direct_estimate}
\big\|D_{u,\lambda_0(u)-1}\big\|_{H^{-1/2,\sqrt{\log}}_+\to\h^{1/2,\sqrt{\log}}_{\ge 0}}<C
\end{equation}
for any $u\in U(v)$.\footnote{In fact, $U(v)$ can be taken to be an open ball in $H^{-1/2,\sqrt{\log}}_{r,0}$ of radius
$1/(8K_0)$ where $K_0>0$ is the constant in Corollary \ref{coro:main_inequality}.}
\item[(ii)] For any $R>0$ there exists a constant $C_R>0$ such that inequality \eqref{eq:pre-Birkhoff_inverse_estimate1}
holds for any $u\in H^{-1/2,\sqrt{\log}}_{r,0}$ that satisfies \eqref{eq:pre-Birkhoff_inverse_estimate2}.
\end{itemize}
\end{Prop}

The following remark is needed for the proof of the properness of the Birkhoff map and its inverse.

\begin{Rem}\label{rem:pre-Birkhoff_estimates}
Since for any $u\in H^{-1/2,\sqrt{\log}}_{r,0}$ the Fourier transform \eqref{eq:K_0} is a isomorphism and 
since $K_{u;-1/2,\sqrt{\log}}|_{L^2_+}\equiv K_{u;0}$ it follows from \eqref{eq:D-operator} and 
\eqref{eq:lambda-estimate*} that for any set $U$ in $H^{-1/2,\sqrt{\log}}_{r,0}$ such that 
$\lambda_0(u)$ is bounded uniformly in $U$ there exists $C>0$ such that
\begin{equation*}
\big\|D_{u,\lambda_0(u)-1}\big\|_{L^2_+\to\h^1_{\ge 0}},
\big\|(D_{u,\lambda_0(u)-1})^{-1}\big\|_{\h^1_{\ge 0}\to L^2_+}<C
\end{equation*}
for any $u\in U$.
\end{Rem}

\begin{Rem}
The identity \eqref{eq:main_identity} and the estimates in Proposition \ref{prop:pre-Birkhoff_estimates} and
Remark \ref{rem:pre-Birkhoff_estimates} can be interpreted as a ``quasi-linearity'' of the pre-Birkhoff map 
$u\mapsto\big(\1 1|f_n(u)\big)_{n\ge 0}$, $H^{-1/2,\sqrt{\log}}_{r,0}\to\h^{1/2,\sqrt{\log}}_{\ge 0}$.
\end{Rem}

\begin{proof}[Proof of Proposition \ref{prop:pre-Birkhoff_estimates}]
For a given $u\in H^{-1/2,\sqrt{\log}}_{r,0}$ we set $\lambda\equiv\lambda_0(u)-1$ in the diagram \eqref{eq:diagram1}
and note that $\lambda_0(u)-1$ is in the resolvent set of $L_u$.
As above, we then choose $u_0\in L^2_{r,0}$ such that $u\in B^{r,0}_{-1/2,\sqrt{\log}}(u_0)$. 
Let us first prove that $\big(\1 1|f_n(u)\2\big)_{n\ge 0}\in\h^{1/2,\sqrt{\log}}_{\ge 0}$.
To this end, note that $L_u\,1=-\Pi u\in H^{-1/2,\sqrt{\log}}_+$, and hence
\[
1\in\widetilde{\dom}(L_u)\,.
\]
We then obtain from the commutative diagram \eqref{eq:diagram1} that
\begin{equation*}
\big(\1 1|f_n(u)\2\big)_{n\ge 0}=K_{u;1/2}\,1=
D_{u,\lambda_0(u)-1}\big(-\Pi u-\lambda_0(u)+1\big)\in\h^{1/2,\sqrt{\log}}_{\ge 0}.
\end{equation*}
This proves \eqref{eq:main_identity} and the fact that $\big(\1 1|f_n(u)\2\big)_{n\ge 0}\in\h^{1/2,\sqrt{\log}}_{\ge 0}$.
Let us now prove \eqref{eq:pre-Birkhoff_direct_estimate}.
Since $D_{u,\lambda_0(u)-1}=(\ell_u-\lambda_0(u)+1)^{-1}K_{u;-1/2,\sqrt{\log}}$ we conclude 
from \eqref{eq:uniform_estimate_norms3}, \eqref{eq:lambda-estimate*}, and $\lambda_0(u)\le 0$, that
there exists $C>0$ such that for any $f\in H^{-1/2,\sqrt{\log}}_+$,
\[
\big\|D_{u,\lambda_0(u)-1}f\|_{\h^{1/2,\sqrt{\log}}_{\ge 0}}
\le C\,\|f\|_{H^{-1/2,\sqrt{\log}}_+}
\]
uniformly in $u\in B^{r,0}_{-1/2,\sqrt{\log}}(u_0)$.
This proves (i). Item (ii) is already proved.
\end{proof}

\section{The Birkhoff map}\label{sec:Birkhoff_map}
In this Section we extend the Birkhoff map \eqref{eq:Phi-introduction} for $u\in H^{-1/2,\sqrt{\log}}_{r,0}$
and prove Theorem \ref{th:main} stated in the Introduction. For simplicity of notation we will identify the (real) space
$\h^{-1/2,\sqrt{\log}}_{r,0}$ with the space $\h^{0,\sqrt{\log}}_+$.

For a given $u\in H^{-1/2,\sqrt{\log}}_{r,0}$ consider the norming constants (cf. \cite[Corollary 3.4]{GK})
\begin{equation}\label{eq:kappa_0}
\kappa_0(u):=\prod_{p\ge 1}\Big(1-\frac{\gamma_p(u)}{\lambda_p(u)-\lambda_0(u)}\Big)
\end{equation}
and 
\begin{equation}\label{eq:kappa_n}
\kappa_n(u):=
\frac{1}{\lambda_n(u)-\lambda_0(u)}\prod_{1\le p\ne n}\Big(1-\frac{\gamma_p(u)}{\lambda_p(u)-\lambda_n(u)}\Big),
\quad n\ge 1.
\end{equation}
The infinite products converge absolutely in view of the absolute convergence in \eqref{eq:trace_formula} and the fact that
$|\lambda_p(u)-\lambda_n(u)|\ge 1$ for $p\ne n$ and $n,p\ge 1$ (cf. \eqref{eq:lambda-inequality}). Note that for any 
$u\in H^{-1/2,\sqrt{\log}}_{r,0}$,
\begin{equation}\label{eq:kappa_n>0}
\kappa_n(u)>0,\quad n\ge 0\,.
\end{equation}
This follows since the infinite products converge and since by \eqref{eq:lambda-inequality},
\[
1-\frac{\gamma_p(u)}{\lambda_p(u)-\lambda_n(u)}=\frac{\lambda_{p-1}(u)-\lambda_n(u)+1}{\lambda_p(u)-\lambda_n(u)}>0
\]
for any $n,p\ge 0$, $p\ne n$.
The following lemma is proved in Appendix \ref{sec:appendix}.

\begin{Lem}\label{lem:kappa-estimate}
For any $n\ge 0$ the norming constant $\kappa_n(u)$ is well-defined and depends continuously on
the potential $u\in H^{-1/2,\sqrt{\log}}_{r,0}$.
For any $v\in H^{-1/2,\sqrt{\log}}_{r,0}$ there exist an open neighborhood $U(v)$ of $v$ in 
$H^{-1/2,\sqrt{\log}}_{r,0}$ and constants $0<c<C$, such that 
\begin{equation}\label{eq:kappa-estimate}
c\le n \kappa_n(u)\le C
\end{equation}
for any $u\in U(v)$ and $n\ge 1$.
\end{Lem}

\begin{Rem}\label{rem:kappa-estimate}
The proof of Lemma \ref{rem:kappa-estimate} shows that the following version of the lemma holds:
Let $U$ be a set in $H^{-1/2,\sqrt{\log}}_{r,0}$ such that the image of the map $U\to\ell^1_+$, 
$u\mapsto\big(\gamma_n(u)\big)_{n\ge 1}$, is a pre-compact set in $\ell^1_+$. Then, there exist constants $0<c<C$ such 
that the inequality \eqref{eq:kappa-estimate} holds for any $u\in U$ and $n\ge 1$. Moreover, $\kappa_0(u)$ is
bounded uniformly for $u\in U$.
\end{Rem}

We can now extend the {\em Birkhoff map} \eqref{eq:Phi-introduction} for $u\in H^{-1/2,\sqrt{\log}}_{r,0}$
by setting (cf. \cite[formula (4.1)]{GK})
\begin{equation}\label{eq:Phi-extended(formulas)}
\Phi(u):=\big(\Phi_n(u)\big)_{n\ge 1},\quad\Phi_n(u):=\frac{\1 1|f_n(u)\2}{\sqrt{\kappa_n(u)}},\quad n\ge 1.
\end{equation}
Recall from \cite[Corollary 3.4]{GK} that for $u\in L^2_{r,0}$ we have that
\begin{equation}\label{eq:gamma-kappa}
\big|\1 1| f_0(u)\2\big|^2=\kappa_0(u)\quad \text{\rm and}\quad
\big|\1 1| f_n(u)\2\big|^2=\gamma_n(u)\kappa_n(u),\quad n\ge 1.
\end{equation}
Since the quantities $\gamma_n(u)$, $\kappa_n(u)$, and $\1 1| f_n(u)\2$, are well-defined and depend continuously
on $u\in H^{-1/2,\sqrt{\log}}_{r,0}$ (cf.  Theorem \ref{th:L_u-real} (iii), Lemma \ref{lem:kappa-estimate}, 
and \eqref{eq:f_n-analyticity}), the relations \eqref{eq:gamma-kappa} continues to hold for $u\in H^{-1/2,\sqrt{\log}}_{r,0}$.
In particular, we obtain from \eqref{eq:Phi-extended(formulas)} and \eqref{eq:gamma-kappa} that
\begin{equation}\label{eq:actions}
\gamma_n(u)=\big|\Phi_n(u)\big|^2,\quad n\ge 1,
\end{equation}
for $u\in H^{-1/2,\sqrt{\log}}_{r,0}$.

\medskip

We are now ready to prove Theorem \ref{th:main} stated in the Introduction. 
We have

\begin{Th}\label{th:Birkhoff_map}
The formula \eqref{eq:Phi-extended(formulas)} defines a map
\begin{equation}\label{eq:Phi-extended} 
\Phi : \quad H^{-1/2,\sqrt{\log}}_{r,0}\to\h^{0,\sqrt{\log}}_+
\end{equation}
that is a homeomorphism. 
\end{Th}

\begin{proof}[Proof of Theorem \ref{th:Birkhoff_map}]
We prove the theorem in several steps:

\medskip

(a) {\em The Birkhoff map \eqref{eq:Phi-extended} is well-defined and injective.}
The fact that \eqref{eq:Phi-extended} is well-defined follows from the first statement of 
Proposition \ref{prop:pre-Birkhoff_estimates} and Lemma \ref{lem:kappa-estimate}.
The injectivity of \eqref{eq:Phi-extended} follows from the explicit formulas in \cite[Lemma 4.2]{GK} and
the arguments in \cite[Proposition 4.3]{GK}. 
Note that all quantities appearing in \cite[Proposition 4.3]{GK} are well-defined for $u\in H^{-1/2,\sqrt{\log}}_{r,0}$.
Denote by $\image(\Phi)\subseteq\h^{0,\sqrt{\log}}_+$ the image of 
\eqref{eq:Phi-extended} and consider the inverse map,
\begin{equation}\label{eq:Phi-inverse}
\Phi^{-1} : \image\big(\Phi\big)\to H^{-1/2,\sqrt{\log}}_{r,0}.
\end{equation}

\medskip

(b) {\em The image of any pre-compact set with respect to the Birkhoff map \eqref{eq:Phi-extended}, and 
its inverse \eqref{eq:Phi-inverse}, is pre-compact.}
Let us first consider the case of the map \eqref{eq:Phi-extended}.
We will follow the arguments in the proof of \cite[Proposition 2 (iii)]{GKT1}. Let $K$ be a pre-compact set in 
$H^{-1/2,\sqrt{\log}}_{r,0}$. Without loss of generality we will assume that $K\subseteq U(v)$ where 
$v\in H^{-1/2,\sqrt{\log}}_{r,0}$ and $U(v)$ is an open neighborhood of $v$ in $H^{-1/2,\sqrt{\log}}_{r,0}$ such that
the statement of Proposition \ref{prop:pre-Birkhoff_estimates} (i), Remark \ref{rem:pre-Birkhoff_estimates}, and
Lemma \ref{lem:kappa-estimate} hold. 
Then, in view of \eqref{eq:main_identity} and \eqref{eq:Phi-extended(formulas)}, there exist a constant $C\equiv C_v>0$
and a linear map $J_u : H^{-1/2,\sqrt{\log}}_+\to\h^{0,\sqrt{\log}}_+$ such that $J_u|_{L^2_+} : L^2_+\to\h^{1/2}_+$
and for any $u\in U(v)$,
\begin{equation}\label{eq:main_identity_direct}
\Phi(u)=J_u\big(\Pi u-\lambda_0(u)+1\big)
\end{equation}
and 
\begin{equation}\label{eq:main_identity_direct_inequalities}
\|J_u\|_{H^{-1/2,\sqrt{\log}}_+\to\h^{0,\sqrt{\log}}_+}\le C,\quad
\|J_u\|_{L^2_+\to\h^{1/2}_+}\le C.
\end{equation}
For any integer $N\ge 0$ consider the projections
\[
\Pi_{\ge N} : H^{-1/2,\sqrt{\log}}_+\to H^{-1/2,\sqrt{\log}}_+,\quad
f\mapsto\sum_{n\ge N}\widehat{f}(n)\,e^{i n x},
\]
and
\[
\Pi_{<N} : H^{-1/2,\sqrt{\log}}_+\to H^{-1/2,\sqrt{\log}}_+,\quad
f\mapsto\sum_{0\le n<N}\widehat{f}(n)\,e^{i n x},
\]
as well as the projections $\pi_{\ge N} : \h^{0,\sqrt{\log}}_+\to \h^{0,\sqrt{\log}}_+$
and $\pi_{<N} : \h^{0,\sqrt{\log}}_+\to \h^{0,\sqrt{\log}}_+$ defined in a similar way.
Now, take $\varepsilon>0$. By Lemma \ref{lem:pre-compact_sets} below, there exists 
an integer $N_\varepsilon\ge 1$ and $R_\varepsilon>0$ such that 
\begin{equation}\label{eq:K-pre-compact}
\big\|\Pi_{\ge N_\varepsilon} u\big\|_{-1/2,\sqrt{\log}}\le\varepsilon/(2C),\quad
\big\|\Pi_{<N_\varepsilon}\big(u-\lambda_0(u)+1\big)\big\|_{L^2_+}\le R_\varepsilon,
\end{equation}
for any $u\in K$. (Note that $\lambda_0(u)$ is uniformly bounded on $K$ since 
$\lambda_0$ depends continuously on $u\in H^{-1/2,\sqrt{\log}}_{r,0}$.)
Then, by \eqref{eq:main_identity_direct}, for any $u\in K$,
\begin{equation*}
\Phi(u)=J_u\Big(\Pi_{<N_\varepsilon}\big(u-\lambda_0(u)+1\big)\Big)+
J_u\big(\Pi_{\ge N_\varepsilon}u\big).
\end{equation*}
This implies that
\begin{equation}\label{eq:Phi-image}
\Phi(K)=\mathcal{I}_1+\mathcal{I}_2
\end{equation}
where
\[
\mathcal{I}_1:=\Big\{J_u\Big(\Pi_{<N_\varepsilon}\big(u-\lambda_0(u)+1\big)\Big)\,\Big|\,u\in K\Big\},\quad
\mathcal{I}_2:=\big\{J_u\big(\Pi_{\ge N_\varepsilon}u\big)\,\big|\,u\in K\big\},
\]
and  $\Phi(K)$ denotes the set $\{\Phi(u)\,|\,u\in K\}$.
It follows from the second inequality in \eqref{eq:main_identity_direct_inequalities} and 
the second inequality in \eqref{eq:K-pre-compact} that the set $\mathcal{I}_1$ is bounded in $\h^{1/2}_+$, and hence, 
it is a {\em compact} set in $\h^{0,\sqrt{\log}}_+$.
Moreover, by the first inequality in \eqref{eq:main_identity_direct_inequalities} and 
the first inequality in \eqref{eq:K-pre-compact}, the set $\mathcal{I}_2$ is contained inside a centered at zero open ball
of radius $\varepsilon/2$ in $\h^{0,\sqrt{\log}}_+$.
By applying Lemma \ref{lem:pre-compact_sets} to the compact set $\mathcal{I}_1$ in $\h^{0,\sqrt{\log}}_+$
and by taking $N_\varepsilon\ge 1$ and $R_\varepsilon>0$ larger if necessary, we obtain from \eqref{eq:Phi-image} that
\begin{equation*}
\big\|\pi_{\ge N_\varepsilon}\Phi(K)\big\|_{\h^{0,\sqrt{\log}}_+}\le
\big\|\pi_{\ge N_\varepsilon}\mathcal{I}_1\|_{\h^{0,\sqrt{\log}}_+}+
\big\|\pi_{\ge N_\varepsilon}\mathcal{I}_2\|_{\h^{0,\sqrt{\log}}_+}
\le\varepsilon
\end{equation*}
and
\begin{equation*}
\big\|\pi_{<N_\varepsilon}\Phi(K)\big\|_{\ell^2_+}\le R_\varepsilon.
\end{equation*}
This and  Lemma \ref{lem:pre-compact_sets} then imply that $\Phi(K)$ is a pre-compact set in $\h^{0,\sqrt{\log}}_+$.
 
Let us now prove that the image of any pre-compact set with respect to \eqref{eq:Phi-inverse} is pre-compact.
Take a pre-compact set $K$ in $\h^{0,\sqrt{\log}}_+$.
For any $u\in\Phi^{-1}(K)$ we have that $\Phi(u)\in K$, and hence, by \eqref{eq:actions}, 
$\big\{\big(\gamma_n(u)\big)_{n\ge 1}\,\big|\,u\in\Phi^{-1}(K)\big\}$ is a pre-compact set in $\ell^1_+$. 
By Remark \ref{rem:kappa-estimate} there exist constants $0<c<C$ such that the inequality \eqref{eq:kappa-estimate} holds 
for any $u\in\Phi^{-1}(K)$ and $n\ge 1$. 
It then follows from \eqref{eq:Phi-extended(formulas)} and the trace formula \eqref{eq:trace_formula} that there exists
a constant $R>0$ such that the condition \eqref{eq:pre-Birkhoff_inverse_estimate2} holds for any $u\in\Phi^{-1}(K)$. 
By Proposition \ref{prop:pre-Birkhoff_estimates}, we then conclude that there exists a constant $C_R>0$ such that
\[
-\Pi u-\lambda_0(u)+1=(D_{u,\lambda_0(u)-1})^{-1}\big(\1 1| f_n(u)\2\big)_{n\ge 0}
\]
where
\[
\big\|(D_{u,\lambda_0(u)-1})^{-1}\big\|_{\h^{1/2,\sqrt{\log}}_{\ge 0}\to H^{1/2,\sqrt{\log}}_+}\le C_R
\]
for any $u\in\Phi^{-1}(K)$.
This and the second inequality in Remark \ref{rem:pre-Birkhoff_estimates} imply that there exist a constant $C_K>0$ and 
a linear map $Q_u : \h^{0,\sqrt{\log}}_{\ge 0}\to H^{-1/2,\sqrt{\log}}_+$ such that 
$Q_u|_{\h^{1/2}_{\ge 0}} : \h^{1/2}_{\ge 0}\to L^2_+$ and for any $u\in\Phi^{-1}(K)$,
\begin{equation}\label{eq:main_identity_direct_inverse}
\Pi u-\lambda_0(u)+1=Q_u\big(\1 1|f_0(u)\2,\Phi(u)\big)
\end{equation}
and 
\begin{equation}\label{eq:main_identity_direct_inequalities_inverse}
\|Q_u\|_{\h^{0,\sqrt{\log}}_{\ge 0}\to H^{-1/2,\sqrt{\log}}_{+}}\le C_K,\quad
\|Q_u\|_{\h^{1/2}_{\ge 0}\to L^2_+}\le C_K.
\end{equation}
Note in addition that by Remark \ref{rem:kappa-estimate} the quantity $|\1 1| f_0(u)\2|^2=\kappa_0(u)$ is
bounded uniformly for $u\in\Phi^{-1}(K)$.
The pre-compactness of $\Phi^{-1}(K)$ then follows from \eqref{eq:main_identity_direct_inverse},
\eqref{eq:main_identity_direct_inequalities_inverse}, and Lemma \ref{lem:pre-compact_sets}, in exactly the same way 
as in the proof of the first part of (b).

\medskip

(c) {\em The Birkhoff map \eqref{eq:Phi-extended} is continuous and onto.}
Since for any given $n\ge 1$ the map \eqref{eq:f_n-analyticity} is continuous, we obtain from Lemma \ref{lem:kappa-estimate}
and \eqref{eq:Phi-extended(formulas)} that for any given $n\ge 1$ the component map 
\begin{equation*}
\Phi_n : H^{-1/2,\sqrt{\log}}_{r,0}\to\C
\end{equation*}
is continuous. Now, we take a sequence $(u_k)_{k\ge 1}$ in $H^{-1/2,\sqrt{\log}}_{r,0}$ that converges to
$u$ in $H^{-1/2,\sqrt{\log}}_{r,0}$. Since the set $\{u_k\,|\,k\ge 1\}$ is pre-compact in  $H^{-1/2,\sqrt{\log}}_{r,0}$,
we conclude from (b) that $\big\{\Phi(u_k)\,\big|\,k\ge 1\big\}$ is pre-compact in $\h^{0,\sqrt{\log}}_+$.
This implies that any subsequence of $\big(\Phi(u_k)\big)_{k\ge 1}$ has a convergent subsequence.
Since $\Phi_n(u_k)\to\Phi_n(u)$ as $k\to\infty$, we then conclude that $\Phi(u_k)\to\Phi(u)$ as $k\to\infty$
in $\h^{0,\sqrt{\log}}_+$.
The ontoness of the Birkhoff map then follows since \eqref{eq:Phi-extended} is continuous, proper, and has 
a dense image in $\h^{0,\sqrt{\log}}_+$.

\medskip

(d) {\em The map $\Phi^{-1} : \h^{0,\sqrt{\log}}_+\to H^{-1/2,\sqrt{\log}}_{r,0}$ is continuous.}
This statement follows from the arguments in (c) and the fact that the components of 
$\Phi^{-1} : \h^{0,\sqrt{\log}}_+\to H^{-1/2,\sqrt{\log}}_{r,0}$ are continuous. 
The latter follows easily from \cite[Lemma 4.2]{GK} and Cauchy's formula.
\end{proof}

In the proof of Theorem \ref{th:Birkhoff_map} we use the following characterization of pre-compact sets in 
$H^{-1/2,\sqrt{\log}}_+$ (and $\h^{0,\sqrt{\log}}_+$). The proof follows easily from 
Cantor's diagonalization process.

\begin{Lem}\label{lem:pre-compact_sets}
A set $K$ is pre-compact in $H^{-1/2,\sqrt{\log}}_+$ if and only if for any $\varepsilon>0$ there exist
an integer $N_\varepsilon\ge 0$ and $R_\varepsilon>0$ such that for any $u\in K$,
\begin{equation*}
\sum_{n\ge N_\varepsilon}\frac{\log(\n+1)}{\n}\,|z_n|^2\le\varepsilon^2,\quad
\sum_{0\le n<N_\varepsilon}|z_n|^2\le R_\varepsilon^2,
\end{equation*}
where $z_n\equiv\widehat{u}(n)$, $n\ge 0$.\\ A similar condition (that involves the weight $\log(\n+1)$
instead of $\frac{\log(\n+1)}{\n}$) characterizes the pre-compact sets in $\h^{0,\sqrt{\log}}_+$.
\end{Lem}

Corollary \ref{coro:well-posedness} follows from the arguments in \cite[Section 5]{GKT1}
(see also \cite[Section 4]{KT1}).

\section{The lack of weak continuity of the flow map and of the Birkhoff map}\label{sec:counterexample}
In this section we prove Proposition \ref{prop:noweakcontinuity}. For this, we revisit the counterexample to 
wellposedness in $H^{-1/2}_{r,0}(\T)$ constructed in \cite{GKT1}, of which we recall the setting.\\
We consider potentials of the form $u_{0,q}(x)=v_q( {\rm e}^{ix} )+\overline {v_q( {\rm e}^{ix} )}$, where $v$ is 
the following Hardy function, defined in the unit disc by
$$
v_q(z)= \frac{\e qz}{1-qz}\ , \quad 0< \e < q<1 \ ,  \qquad  |z|<1\ .
$$
Note that
\begin{equation}\label{normofu}
\| u_{0,q}\|_{-1/2,\sqrt{\log }}^2=2\e ^2\sum_{n=1}^\infty n^{-1}\log (1+n)\, q^{2n}\sim \e ^2(\log (1-q))^2\ ,
\end{equation}
as $q$ tends to $1$. We choose
\begin{equation}\label{epsq}
\e =\frac{\beta }{|\log (1-q)|}
\end{equation}
where $\beta >0$ is a positive parameter which will be fixed later. Therefore we have
$\| u_{0,q}\|_{-1/2,\sqrt{\log }}\to\beta $ as $q\to 1$, and $u_{0,q}$ tends weakly to $0$ in $H^{-1/2,\sqrt{\log}}_{r,0}$.
The study of the Lax operator $L_{u_{0,q}}$ reduces to the study of a first order linear differential equation in the complex domain,
which is processed in \cite{GKT1}. From this analysis, we infer that $-\mu $ is a negative eigenvalue of $L_{u_{0,q}}$ if and 
only if $F(\mu, q)=0$, where
\begin{eqnarray*}
F(\mu, q)&=&F_+(\mu, q)-F_-(\mu, q)\ ,\\
F_+(\mu, q)&:=&\int_0^q \frac{\mu t^{\e +\mu}(1-qt)^\e }{t(q-t)^\e }\, dt\ ,\\ 
F_-(\mu, q)&:=&\int_0^q \frac{ \e qt^{\e +\mu}(1-qt)^\e }{(q-t)^\e (1-qt) }\, dt \ .
\end{eqnarray*}
and where we recall that $\e $ is given by \eqref{epsq}.
Moreover, $F(\mu, q)>0$ for $\mu =\e q^2/(1-q^2)$, and, as $q\to 1$, for every fixed $\mu >0$,
$$F_+(\mu ,q)\to 1\ ,\ F_-(\mu ,q)\sim -\e \log (1-q^2)\to \beta \ .$$
Consequently, $F(\mu, q)\to 1-\beta $ as $q\to 1$. Let us now choose $\beta >1$. Then we infer that $F(\mu, q)$ must vanish 
for some $\mu _q$ tending to $+\infty $ as $q$ tends to $1$. Furthermore, since $\partial_\mu F(\mu, q)>0$ if $F(\mu, q)=0$, 
we know that such a zero $\mu_q$ is unique. We conclude that $L_{u_{0,q}}$ has a unique negative eigenvalue
$\lambda_0(u_{0,q})=-\mu_q$, and that this eigenvalue tends to $-\infty$. Consequently,
$$\gamma_1(u_{0,q})=\lambda_1(u_{0,q})-\lambda_0(u_{0,q})\to +\infty $$
and therefore the function $\gamma_1$ is not weakly continuous on $H^{-1/2,\sqrt{\log}}_{r,0}$.
A fortiori, $\Phi :H^{-1/2,\sqrt{\log}}_{r,0}\to \h^{0,\sqrt{\log}}_+$ is not weakly continuous.\\
Finally, we prove that the flow map is not weakly continuous in the same way than in \cite{GKT1}. 
Denote by $u_q$  the Benjamin--Ono solution with the initial datum $u_{0,q}$. 
Then it is proved in \cite{GKT1} that the function $$\xi_q(t)=\langle u_q(t)\vert {\rm e}^{ix}\rangle $$
is bounded and satisfies, for every finite interval $I$,
$$\left |\int_I \xi_q(t){\rm e}^{-it(1-2\mu_q)}\, dt\right | =\sqrt 2 |I|+O\left (\frac{1}{\mu_q}\right )\ .$$
Hence $\xi_q(t)$ cannot tend to $0$ on any  time interval of positive length. 
This completes the proof of Proposition \ref{prop:noweakcontinuity}.

\section{The convolution in log-spaces}\label{sec:composition}
In this section, we discuss basic properties of the convolution in the spaces with logarithmic weights and
prove Proposition \ref{prop:no_analyticity} formulated in the Introduction.

\medskip

We will first prove the following auxiliary lemma.

\begin{Lem}\label{lem:convolution}
There exist $(x_n)_{n\in\Z}\in\h^{-1/2,\sqrt{\log}}_{r,0}$ and  $(y_n)_{n\in\Z}\in\h^{1/2,\sqrt{\log}}_{r,0}$
such that the sequence  $z_n:=\sum_{k\ge 0,k\ne n}x_ky_{n-k}$, $n\ge 1$,
does {\em not} belong to $\h^{-1/2,\sqrt{\log}}_+$.
\end{Lem}

\begin{proof}[Proof of Lemma \ref{lem:convolution}]
Assume that $(x_n)_{n\in\Z}\in\h^{-1/2,\sqrt{\log}}_{r,0}$, $(y_n)_{n\in\Z}\in\h^{1/2,\sqrt{\log}}_{r,0}$,
and $z_n:=\sum_{k\ge 0,k\ne n}x_ky_{n-k}$ for $n\ge 1$.
Then, we can write
\[
x_n\!=\!\frac{\sqrt{\n}}{\sqrt{\log(\n+1)}}\,a_n,\,
y_n\!=\!\frac{1}{\sqrt{\n}\sqrt{\log(\n+1)}}\,b_n,\,
z_n\!=\!\frac{\sqrt{\n}}{\sqrt{\log(\n+1)}}\,c_n
\]
where $a:=(a_n)_{n\in\Z}\in\ell^2_{r,0}$, $b:=(b_n)_{n\in\Z}\in\ell^2_{r,0}$ and 
$c:=(c_n)_{n\ge 1}$ is a complex-valued sequence.
Since $z_n=\sum_{k\ge 0,k\ne n}x_ky_{n-k}$ we obtain that for $n\ge 1$,
\begin{equation}\label{eq:c_n}
c_n=\sum_{k\ge 0,k\ne n}a_k\frac{b_{n-k}}{\sqrt{\1 n-k\2\log(\1 n-k\2+1)}}\,B_{n,k}
\end{equation}
where 
\begin{equation}\label{eq:B_n,l}
B_{n,k}:=\left(\frac{\1 k\2\log(\n+1)}{\n\log(\1 k\2+1)}\right)^{1/2}.
\end{equation}
The lemma will follow once we construct $(a_n)_{n\in\Z}, (b_n)_{n\in\Z}\in\ell^2_{r,0}$ such that 
$(c_n)_{n\ge 1}\notin\ell^2_+$. Assume that the elements of the sequences $a,b\in\ell^2_{r,0}$
are chosen real-valued and non-negative,
\[
a_n\ge 0,\quad b_n\ge 0,\quad n\in\Z\,.
\]
Note that the sequence $\big(\1 k\2/\log(\1 k\2+1)\big)_{k\ge 1}$ is monotone increasing. 
This together with \eqref{eq:B_n,l} implies that there exists a constant $C>0$ such that 
for any $n\ge 1$ and $n/2<k<n$ we have that
\[
B_{n,k}\ge B_{n,[n/2]}\ge C>0
\]
where $[n/2]$ denotes the integer part of $n/2$.
We then obtain from \eqref{eq:c_n} that
\begin{align}
\|c\|_{\ell^2_+}^2&\ge
\sum_{n\ge 1}\Big(\sum_{n/2<k<n}a_k\frac{b_{n-k}}{\sqrt{\1 n-k\2\log(\1 n-k\2+1)}}\,B_{n,k}\Big)^2\nonumber\\
&\ge C\sum_{n\ge 1}\Big(\sum_{n/2<k<n}a_k\frac{b_{n-k}}{\sqrt{\1 n-k\2\log(\1 n-k\2+1)}}\Big)^2.\label{eq:c-estimate1}
\end{align}
Now, assume that the sequence $a\in\ell^2_{r,0}$ is chosen so that $(a_n)_{n\ge 1}$ is monotone decreasing. 
Then, in view of \eqref{eq:c-estimate1},
\begin{equation}\label{eq:c-estimate2}
\|c\|_{\ell^2_+}^2\ge C\sum_{n\ge 1}a_n^2\Big(\sum_{0<l<n/2}\frac{b_l}{\sqrt{\1 l\2\log(\1 l\2+1)}}\Big)^2
\end{equation}
where we passed to the index $l:=n-k$ in the internal sum.
By choosing $b_0:=0$ and
\[
b_l:=\frac{1}{\sqrt{\1 l\2\log(\1 l\2+1)}\big(\log(\log(\1 l\2+1))\big)^{3/4}},\quad  |l|\ge 1,
\]
we see that $b\in\ell^2_{r,0}$ and by the integral test's estimate
\begin{align}
\sum_{0<l<n/2}\frac{b_l}{\sqrt{\1 l\2\log(\1 l\2+1)}}
&=\sum_{0<l<n/2}\frac{1}{\1 l\2\log(\1 l\2+1)\big(\log(\log(\1 l\2+1)\big)^{3/4}}\nonumber\\
&\ge C_1\big(\log(\log(\n+1))\big)^{1/4}\label{eq:c-estimate3}
\end{align}
for some positive constant $C_1>0$ independent of $n\ge 1$.
Hence, by \eqref{eq:c-estimate2} and \eqref{eq:c-estimate3},
\begin{equation}\label{eq:c-estimate4}
\|c\|_{\ell^2_+}^2\ge C_2\sum_{n\ge 1}\Big(\big(\log(\log(\n+1))\big)^{1/4}a_n\Big)^2
\end{equation}
for some constant $C_2>0$ independent of $n\ge 1$.
If we now choose $a_0:=0$ and
\begin{equation}\label{eq:a_n}
a_n:=b_n\equiv\frac{1}{\sqrt{\n\log(\n+1)}\big(\log(\log(\n+1))\big)^{3/4}},\quad |n|\ge 1,
\end{equation}
we obtain that $a\in\ell^2_{r,0}$ and the series on the right hand side of \eqref{eq:c-estimate4}
diverges by the integral test. This completes the proof of the lemma.
\end{proof}

For the proof of Proposition \ref{prop:no_analyticity} we will need the following 
variant of Lemma \ref{lem:convolution}. For $s>-1/2$ consider the quadratic form
\begin{equation}\label{eq:Q}
\h^s_{r,0}\to\h^s_+,\quad x\mapsto 
Q(x):=\Big(\frac{1}{\sqrt{n}}\sum_{k\ge 0,k\ne n}x_k\frac{x_{n-k}}{n-k}\Big)_{n\ge 1}.
\end{equation}
The quadratic form \eqref{eq:Q} is well-defined and bounded by Lemma 3.1 in \cite{GKT3}.

\begin{Lem}\label{lem:Q}
There exists $x\in\h^{-1/2,\sqrt{\log}}_{r,0}$ such that $Q(x)\in\ell^2_+$ but 
$Q(x)$ does \em{not} belong to $\h^{0,\sqrt{\log}}_+$ .
\end{Lem}

\begin{proof}[Proof of Lemma \ref{lem:Q}]
We set $x_0:=0$ and
\[
x_n:=\frac{\sqrt{\n}}{\sqrt{\log(\n+1)}}\,a_n=\frac{1}{\log(\n+1)\big(\log(\log(\n+1))\big)^{3/4}}\quad |n|\ge 1,
\]
where $(a_n)_{n\in\Z}$ is the sequence \eqref{eq:a_n} from the proof of Lemma \ref{lem:convolution}. 
The fact that the sequence $x:=(x_n)_{n\in\Z}$ satisfies the conditions of the lemma follows easily from the proof
of Lemma \ref{lem:convolution}.
\end{proof}

\begin{Coro}\label{coro:Q}
The quadratic form \eqref{eq:Q} cannot be extended to a bounded quadratic form
$Q : \h^{-1/2,\sqrt{\log}}_{r,0}\to\h^{0,\sqrt{\log}}_+$.
\end{Coro}

On a side note, let us also mention that the arguments in the proof of Lemma \ref{lem:convolution} above imply
that in contrast to the boundedness of the maps \eqref{eq:T,L,s>-1/2} we have

\begin{Coro}\label{coro:Toeplitz_operator}
There exist $u\in H^{-1/2,\sqrt{\log}}_{r,0}$ and $f\in H^{1/2,\sqrt{\log}}_+$ such that
$T_u f\in H^{-1/2}_+$ but $T_u f\notin H^{-1/2,\sqrt{\log}}_+$.
\end{Coro}

Let us now compute the second differential of the Birkhoff map \eqref{eq:Phi-introduction} at $u=0$. 
For simplicity of notation we identify the (real) space $\h^\beta_{r,0}$ with $\h^\beta_+$, $\beta\in\R$, 
and write 
\begin{equation}\label{eq:Phi-reduced}
\Phi : H^s_{r,0}\to\h^{\frac{1}{2}+s}_+,\ u\mapsto
\big(\Phi_n(u)\big)_{n \ge 1},\quad s>-1/2\,.
\end{equation}
Recall from \cite{GK} and \cite[formula (93)]{GKT3} that the differential 
$d_0\Phi : H^s_{r,0}\to\h^{\frac{1}{2}+s}_+$ of \eqref{eq:Phi-reduced} at $u=0$ coincides 
with the weighted Fourier transform $\xi\mapsto\left(-\frac{\widehat{\xi}(n)}{\sqrt{n}}\right)_{n\ge 1}$.
For the second differential $d^2_0\Phi$ of \eqref{eq:Phi-reduced} at $u=0$ we have

\begin{Lem}\label{lem:second_differential}
For $s>-1/2$ and for any $\xi\in H^s_{r,0}$ ,
\begin{equation}\label{eq:second_differential}
d^2_0\Phi(\xi)=
\Big(-\frac{1}{\sqrt{n}}\sum_{k\ge 0,k\ne n}\widehat{\xi}(-k)\frac{\widehat{\xi}(k-n)}{k-n}\Big)_{n\ge 1}\,.
\end{equation}
\end{Lem}

\begin{proof}[Proof of Lemma \ref{lem:second_differential}]
We will follow the framework developed in \cite[Section 4]{GKT3}.
Assume that $s>-1/2$. By \cite[formula (58)]{GKT3} for any $u$ in an open neighborhood $U$ of zero 
in $H^s_{r,0}$,
\begin{equation}\label{eq:Phi_n}
\Phi_n(u)=-\sqrt{n}\,\frac{a_n(u)}{\sqrt[+]{n\kappa_n(u)}}\,\Psi_n(u),\quad n\ge 1,
\end{equation}
where $\Psi_n(u) : U\to\C$, $\kappa_n : U\to\R$, and $a_n : U\to\C$, are analytic maps
(cf. \cite[Proposition 3.1]{GKT3}, \cite[Lemma 4.1, Lemma 4.3]{GKT3}) and
$\sqrt[+]{\cdot}$ denotes the branch of the square root defined by $\sqrt[+]{1}=1$.
Here $\Psi_n(u):=\1 h_n(u),1\2$ is the $n$-th component of the pre-Birkhof map studied in \cite[Section 3]{GKT3},
$h_n(u):=P_n(u) e_n$ where $P_n\equiv P_n(u)$ is the Riesz projector onto the $n$-th eigenspace of 
the Lax operator $L_u\equiv D-T_u$, and $e_n:=e^{i n x}$, $n\ge 0$ (cf. \cite[formula (19)]{GKT3}). 
Since $h_n(0)=e^{i nx}$, $n\ge 0$, we conclude that $\Psi_n(0)=0$, $n\ge 1$.
This together with (26) and (29) in \cite{GKT3} implies that for any $\xi\in H^s_{r,0}$ and $n\ge 1$,
\begin{equation}\label{eq:Psi_n-differential}
\Psi_n(0)=0,\quad d_0\Psi_n(\xi)=\frac{\widehat{\xi}(-n)}{n},\quad
d^2_0\Psi_n(\xi)=-\frac{1}{n}\sum_{k\ge 0,k\ne n}\widehat{\xi}(-k)\frac{\widehat{\xi}(k-n)}{k-n}.
\end{equation}
The norming constants $\kappa_n(u)$, $n\ge 0$, are given by the product representation (34) in \cite{GKT3},
$\kappa_n(u)>0$ for $u\in U$, and (see \cite[Remark 5.2]{GK}, \cite[Corollary 6 (iv)]{GKT2})
\begin{equation}\label{eq:kappa_n-differential}
\sqrt[+]{n\kappa_n(0)}=1,\quad d_0\kappa_n=0,\quad n\ge 0.
\end{equation}
We will also need the norming constants $\mu_n(u)>0$, $n\ge 1$, $u\in U$, given by the product representation 
(see e.g. \cite{GK}, \cite[formula (35)]{GKT3})
\begin{equation}\label{eq:mu_n}
\mu_n:=\Big(1-\frac{\gamma_n}{\lambda_n-\lambda_0}\Big)
\prod_{k\ge 1,k\ne n}\Big(1-\gamma_n\frac{\gamma_k}{(\lambda_{k-1}-\lambda_{n-1})(\lambda_k-\lambda_n)}\Big)
\end{equation}
where $\gamma_n\equiv\gamma(u):=\lambda_n(u)-\lambda_{n-1}(u)-1\ge 0$ are the spectral gaps and 
$\lambda_n\equiv\lambda_n(u)$, $n\ge 0$, are the eigenvalues of the Lax operator $L_u$.
Since the product \eqref{eq:mu_n} converges absolutely and locally uniformly on $U$ 
(cf. \cite[Theorem 3]{GKT2}) we can differentiate it term by term to conclude from $\lambda_n(0)=n$, 
$d_0\gamma_n=0$, $n\ge 1$ (\cite[Remark 5.2]{GK}) that
\begin{equation}\label{eq:mu_n-differential}
\mu_n(0)=1,\quad d_0\mu_n=0,\quad n\ge 1.
\end{equation}
Let us now turn our attention to the quantities $a_n(u)$, $n\ge 1$, defined recursively for $u\in U$ by
\begin{equation}\label{eq:a_n(u)}
a_0(u):=\frac{\sqrt[+]{\kappa_0(u)}}{\1 h_0(u),1\2},\quad 
a_n(u)=\frac{\nu_n(u)}{\sqrt[+]{\mu_n(u)}}\,a_{n-1}(u),\quad n\ge 1,
\end{equation}
where
\begin{equation}\label{eq:nu_n}
\nu_n(u):=1+\frac{\delta_n(u)}{\alpha_n(u)},\,\,\,\alpha_n(u):=\1P_n e_n|e_n\2,\,\,\,
\beta_n(u):=\1 P_nSP_{n-1}e_{n-1}|e_n\2,
\end{equation}
\begin{equation}\label{eq:delta_n}
\delta_n(u):=\beta_n(u)-\alpha_n(u),
\end{equation}
and $S : H^{s+1}_+\to H^{s+1}_+$ is the shift operator (cf. \cite[Section 4]{GKT3}).
Since $\alpha_n(0)=\beta_n(0)=\1 e_n|e_n\2=1$ we conclude from \eqref{eq:delta_n} that $\delta_n(0)=0$, $n\ge 1$.
By combining this with the first formula in \eqref{eq:nu_n} we obtain that
\begin{equation}\label{eq:nu_n-differential}
\nu_n(0)=1,\quad d_0\nu_n=d_0\delta_n,\quad n\ge 1.
\end{equation}
It follows from \eqref{eq:mu_n-differential}, \eqref{eq:nu_n-differential}, and \eqref{eq:a_n(u)} that
\begin{equation}\label{eq:a_n(0)}
a_n(0)=1,\quad n\ge 0\,.
\end{equation}
In order to compute the differential $d_0a_0$ consider the Taylor's expansion of $\Psi_0(u):=\1 h_0(u),1\2$ for $u\in U$ at zero 
\begin{align}
\Psi_0(u)&\equiv\1 P_0(u) 1,1\2=-\frac{1}{2\pi i}\oint_{\partial D_0}\1(L_u-\lambda)^{-1} 1,1\2\,d\lambda\nonumber\\
&=\frac{1}{2\pi i}\sum_{m\ge 1}\oint_{\partial D_0}\big\1[T_u(D-\lambda)^{-1}]^m 1,1\big\2\,\frac{d\lambda}{\lambda}\nonumber\\
&=-\frac{1}{2\pi i}\oint_{\partial D_0}\frac{\1 u,1\2}{\lambda^2}\,d\lambda+\cdots\label{eq:Phi_0-expansion}
\end{align}
where $\cdots$ stands for terms of order $\ge 2$ in $u$ and $\partial D_0$ is the counterclockwise oriented boundary 
of the centered at zero closed disk of radius $1/3$ in $\C$ and the neighborhood $U$ is chosen as in \cite[Proposition 2.2]{GKT3}.
Since the integral in \eqref{eq:Phi_0-expansion} vanishes we conclude that
\[
\Psi_0(0)=\1 P_0(0) 1,1\2=1,\quad d_0\Psi_0=0.
\]
By combining this with \eqref{eq:kappa_n-differential} we obtain from the first formula in \eqref{eq:a_n(u)} that 
\begin{equation}\label{eq:a_0-differential}
d_0a_0=0.
\end{equation}
It follows from \eqref{eq:mu_n-differential}, \eqref{eq:nu_n-differential}, \eqref{eq:a_n(0)}, 
and the second formula in \eqref{eq:a_n(u)} that
\begin{equation*}
d_0a_n=d_0\delta_n+d_0a_{n-1},\quad n\ge 1.
\end{equation*}
Hence, we conclude from \eqref{eq:a_0-differential} that
\begin{equation}\label{eq:a_n-differential'}
d_0a_n=\sum_{1\le k\le n}d_0\delta_k,\quad n\ge 1.
\end{equation}
In order to compute $d_0\delta_n$, $n\ge 1$, we argue as follows.
Recall from \cite[Section 5]{GKT3} that the Taylor's expansion of $\delta_n(u)$ for $u\in U$ at zero is 
given by \cite[formula (69)]{GKT3}. This implies that
\begin{equation}\label{eq:delta_n-differential'}
d_0\delta_n(u)=-\sum_{k\ge 0}\sigma_{n,k}C_{u,n}(k),\quad n\ge 1,
\end{equation}
where
\begin{equation}\label{eq:C_u,n}
C_{u,n}(k):=
\frac{1}{2\pi i}\oint_{\partial D_{n-1}}\!\!\!\!\frac{\widehat{u}\big(k-(n-1)\big)}{(n-1)-\lambda}\frac{d\lambda}{k-\lambda}
=\left\{
\begin{array}{lc}
\frac{\widehat{u}(k-(n-1))}{k-(n-1)}&,\, k\ne n-1,\\
0&,\, k=n-1,
\end{array}
\right.
\end{equation}
and $\sigma_{n,k}$ is the term of order zero in $u$ in the expansion of $\1 P_n(u)S e_k|e_n\2$ for $u\in U$ at zero
\begin{align}
\1 P_n(u)S e_k|e_n\2&=-\frac{1}{2\pi i}\sum_{r\ge 0}\oint_{\partial D_n}
\big\1(D-\lambda)^{-1}[T_u(D-\lambda)^{-1}]^r Se_k\big|e_n\big\2\,d\lambda\nonumber\\
&=-\frac{1}{2\pi i}\oint_{\partial D_n}\frac{\1 e_{k+1}|e_n\2}{(k+1)-\lambda}\,d\lambda+\cdots,\quad n\ge 1,
\end{align}
where $\cdots$ stands for terms of order $\ge 1$ in $u$ and $\partial D_n$ is the counterclockwise oriented boundary 
of the centered at $n$ closed disk of radius $1/3$ in $\C$.
This implies that
\begin{equation*}
\sigma_{n,k}=-\frac{1}{2\pi i}\oint_{\partial D_n}\frac{\1 e_{k+1}|e_n\2}{(k+1)-\lambda}\,d\lambda=\delta_{k,n-1}.
\end{equation*}
By combining this with \eqref{eq:delta_n-differential'} and \eqref{eq:C_u,n} we obtain that $d_0\delta_n=0$, $n\ge 1$.
Hence, by \eqref{eq:a_n-differential'},
\begin{equation}\label{eq:a_n-differential(b)}
d_0 a_n=0,\quad n\ge 1.
\end{equation}
Finally, the expression \eqref{eq:second_differential} for the second differential of \eqref{eq:Phi-reduced}
follows from the product rule applied twice to \eqref{eq:Phi_n} together with \eqref{eq:kappa_n-differential}, 
\eqref{eq:a_n(0)}, and \eqref{eq:a_n-differential(b)}.
\end{proof}

\begin{proof}[Proof of Proposition \ref{prop:no_analyticity}]
The proposition follows directly from Corollary \ref{coro:Q} and Lemma \ref{lem:second_differential}.
In fact, take $s>-1/2$ and assume that the Birkhoff map \eqref{eq:Phi-reduced} extends to a $C^2$-map
\begin{equation}
\Phi : H^{-1/2,\sqrt{\log}}_{r,0}\to\h^{0,\sqrt{\log}}_+.
\end{equation}
Then, its second differential at zero $d_0^2\Phi : H^{-1/2,\sqrt{\log}}_{r,0}\to\h^{0,\sqrt{\log}}_+$ 
is a bounded extension of the second differential \eqref{eq:second_differential} of the map \eqref{eq:Phi-reduced}.
Since this contradicts Corollary \ref{coro:Q} we conclude that the map \eqref{eq:Phi-reduced} cannot be
extended to a $C^2$-map and, in particular, to an analytic map.
\end{proof}

\appendix

\section{Auxiliary results}\label{sec:appendix}
In this Appendix we provide the proofs of several technical results used in the main body of the paper.
We start with the following lemma on the pointwise multiplication of functions in $H^{1/2}_c$.

\begin{Lem}\label{lem:multiplication_in_H^{1/2}}
For any $u,v\in H^{1/2}_c$ we have that $uv\in H^{1/2,1/\sqrt{\log}}_c$ and the map
\[
H^{1/2}_c\times H^{1/2}_c\to H^{1/2,1/\sqrt{\log}}_c,\quad(u,v)\mapsto uv,
\]
is bounded.
\end{Lem}

\begin{proof}[Proof of Lemma \ref{lem:multiplication_in_H^{1/2}}]
The lemma easily follows by using the dyadic decomposition of functions (see e.g. \cite[Chapter II]{AlinhacGerard}).
Below we give the proof for the reader's convenience. For $f\in\mathcal{D}'(\T)$ we set 
\[
f_{-1}:=\widehat{f}(0),\quad
f_n:=\sum_{2^{n-1}<|k|<2^{n+1}}\varphi\big(k/2^n\big)\,\widehat{f}(k)\,e^{i k x},\quad n\ge 0,
\]
where $\varphi(\xi):=\psi(\xi/2)-\psi(\xi)$, $\psi\in C^\infty_c(\R)$ has non-negative values, $\psi(\xi)=1$ for $|\xi|\le1/2$, 
and $\psi(\xi)=0$ for $|\xi|\ge 1$. Then, the functions $\psi(\xi)$ and $\varphi(\xi/2^n)$, $n\ge 0$, provide a partition of unity
of $\R$. As in the case on the line one then sees that $f\in H^s_c$, $s\in\R$, if and only if
$\Big(2^{ns}\,\|f_n\|\Big)_{n\ge -1}\in\ell^2_{\ge -1}$. The norm on $H^s_c$
and the norm $\|f\|_s:=\Big(\sum_{n\ge -1}2^{2ns}\,\|f_n\|^2\Big)^{1/2}$ are equivalent. 
Similarly, $f\in H^{s,1/\sqrt{\log}}_c$ with $s\in\R$ if and only if
$\Big(\frac{2^{ns}}{\sqrt{\n}}\,\|f_n\|\Big)_{n\ge -1}\in\ell^2_{\ge -1}$, and the corresponding norms are
equivalent. For $u,v\in H^{1/2}_c$ we write
\begin{equation}\label{eq:decomposition}
uv=\sum_{m,n\ge -1}u_nv_m=\sum_{n\ge -1}(S_nu)v_n+\sum_{m\ge -1}u_m(S_{m+1}v)
\end{equation}
where $S_n f:=\sum_{-1\le k\le n-1} f_k$. We have
\begin{align*}
\|S_nu\|_{L^\infty}&\le\sum_{|k|\le 2^n}|\widehat{u}(k)|
\le C_1\,\Big(\sum_{|k|\le 2^n}\frac{1}{k}\,\,\,\Big)^{1/2}\|u\|_{1/2}\nonumber\\
&\le C_2\sqrt{\n}\,\|u\|_{1/2}
\end{align*}
with constants $C_1, C_2>0$ independent of $n\ge -1$.
Hence,
\[
\|(S_nu)v_n\|\le \|S_nu\|_{L^\infty}\|v_n\|\le C_2\,\|u\|_{1/2}\sqrt{\n}\,\|v_n\|
\]
and, from the dyadic characterization of $H^{1/2}_c$,
\[
\frac{2^{n/2}}{\sqrt{\n}}\,\|(S_nu)v_n\|\le C_3\,\|u\|_{1/2}\,c_n
\]
where $\sum_{n\ge -1} c_n^2=1$ and $C_3>0$ is independent of $n\ge -1$. 
By arguing as in the proof of \cite[Lemma 2.1]{AlinhacGerard} we then conclude that 
the first sum on the right side of \eqref{eq:decomposition} belongs to $H^{1/2,1/\sqrt{\log}}_c$ and 
\[
\big\|\sum_{n\ge -1}(S_nu)v_n\big\|_{1/2,1/\sqrt{\log}}\le C\,\|u\|_{1/2}\|v\|_{1/2}
\]
with a constant $C>0$ independent of the choice of $u,v\in H^{1/2}_c$.
The second sum on the right side of \eqref{eq:decomposition} is treated in the same way.
This completes the proof of the lemma.
\end{proof}

As a corollary from Lemma \ref{lem:multiplication_in_H^{1/2}} we obtain the following

\begin{Coro}\label{coro:main_inequality}
For any $u\in H^{-1/2,\sqrt{\log}}_c$ and $v\in H^{1/2}_c$ we have that $uv\in H^{-1/2}_c$ and the map
\[
H^{-1/2,\sqrt{\log}}_c\times H^{1/2}_c\to H^{-1/2}_c,\quad(u,v)\mapsto uv,
\]
is bounded. In particular, there exists a positive constant $K_0>0$ such that
$\|uv\|_{-1/2}\le K_0\|u\|_{-1/2,\sqrt{\log}}\|v\|_{1/2}$ for any $u\in H^{-1/2,\sqrt{\log}}_c$ and $v\in H^{1/2}_c$.
\end{Coro}

\noindent The corollary follows easily by duality form Lemma \ref{lem:multiplication_in_H^{1/2}}.

\medskip

\begin{proof}[Proof of Lemma \ref{lem:kappa-estimate}]
Recall from \eqref{eq:trace_formula} that for any $u\in H^{-1/2,\sqrt{\log}}_{r,0}$,
\begin{equation}\label{eq:trace_formula'}
\sum_{p=1}^\infty\gamma_p(u)=-\lambda_0(u),
\end{equation}
where $\lambda_0(u)$ and $\gamma_p(u)\ge 0$, $p\ge 1$, depend continuously on 
$u\in H^{-1/2,\sqrt{\log}}_{r,0}$ (Theorem \ref{th:L_u-real} (iii)).
Then, by Dini's theorem (see, e.g., \cite[Theorem 8, Ch. 4]{Widom}), the series in \eqref{eq:trace_formula'} converges
uniformly on compact sets of $u$'s in $H^{-1/2,\sqrt{\log}}_{r,0}$. The continuity of \eqref{eq:kappa_0} and
\eqref{eq:kappa_n} then follows from the uniform convergence of the infinite products and the continuity of the quantities involved.
Let us now prove \eqref{eq:kappa-estimate}.
For any $u\in H^{-1/2,\sqrt{\log}}_{r,0}$ and $n\ge 1$ we have
\begin{align}
\sum_{p\ge 1, p\ne n}\frac{\gamma_p(u)}{|\lambda_p(u)-\lambda_n(u)|}&\le
\sum_{|n-p|>\frac{n}{2}}\frac{\gamma_p(u)}{|p-n|}+
\sum_{|n-p|\le\frac{n}{2}, p\ne n}\frac{\gamma_p(u)}{|p-n|}\nonumber\\
&\le\frac{2}{n}\big(-\lambda_0(u)\big)+\sum_{p\ge\frac{n}{2}}\gamma_p(u)\label{eq:kappa-estimate1}
\end{align} 
where we use that $|\lambda_p-\lambda_n|\ge|p-n|$ by \eqref{eq:lambda-inequality}.
Let us now pick $v\in H^{-1/2,\sqrt{\log}}_{r,0}$ and $\varepsilon>0$.
Then, we can choose $n_0\ge 1$ so that $\frac{2}{n_0}\big(-\lambda_0(v)\big)\le\varepsilon/4$ and 
$\sum_{p\ge\frac{n_0}{2}}\gamma_p(v)\le\varepsilon/4$. 
By the continuity of $\lambda_0(u)$ and $\sum_{p\ge\frac{n_0}{2}}\gamma_p(u)$ with respect to
$u\in H^{-1/2,\sqrt{\log}}_{r,0}$ we then obtain that there exists an open neighborhood $U(v)$ of $v$ in 
$H^{-1/2,\sqrt{\log}}_{r,0}$ such that the expression on the right side of \eqref{eq:kappa-estimate1}
is bounded above by $\varepsilon$ uniformly in $u\in U(v)$ and $n\ge n_0$. This proves that 
\begin{equation}\label{eq:series-estimate1}
\sum_{p\ge 0, p\ne n}\frac{\gamma_p(u)}{|\lambda_p(u)-\lambda_n(u)|}\le\varepsilon
\end{equation}
for any $u\in U(v)$ and $n\ge n_0$. The existence of the constants $0<c<C$ and the estimate \eqref{eq:kappa-estimate} 
for $n\ge n_0$ then follows from \eqref{eq:series-estimate1}, \eqref{eq:kappa_n},
and the continuous dependence of the eigenvalues on the potential $u\in H^{-1/2,\sqrt{\log}}_{r,0}$.\\
The case $1\le n<n_0$ follows from \eqref{eq:kappa_n>0} and the continuous dependence of $\kappa_n(u)$ on
$u\in H^{-1/2,\sqrt{\log}}_{r,0}$.
\end{proof}


\begin{thebibliography}{99}

\bibitem{AlinhacGerard} S. Alinhac, P. G\'erard, {\em Pseudo-differential operators and the Nash-Moser theorem}
Graduate Studies in Mathematics, $\bf 82$, AMS, 2007

\bibitem{Ben1967} T. Benjamin, {\em Internal waves of permanent form in fluids of great depth},
J. Fluid Mech., $\bf 29$(1967), 559-592


\bibitem{Birk} G. Birkhoff, {\em Dynamical Systems}, AMS, 1927

\bibitem{BK1979} T. Bock, M. Kruskal, {\em A two parameter Miura transform of the Benjamin-Ono equation},
Phys. Lett. A, $\bf 74$(1979), 173-176

\bibitem{DA1967} R. Davis, A. Acrivos, {\em Solitary internal waves in deep water}, 
J. Fluid Mech. $\bf 29$(1967), 593-607


\bibitem{G} P.~G\'erard,  {\em An explicit formula for the Benjamin--Ono equation}, arXiv: 2212.03139, to appear in Tunisian Journal of Mathematics.

\bibitem{GK} P. G\'erard, T. Kappeler, {\em On the integrability of the Benjamin-Ono equation},
Commun. Pure Appl. Math., $\bf 74$(2021), no. 8,1685-1774

\bibitem{GKT1} P. G\'erard, T. Kappeler, P. Topalov, {\em Sharp well-posedness results of the Benjamin-Ono equation in
$H^{s}(\T,\R)$ and qualitative properties of its solutions}, to appear in Acta Mathematica, arXiv: 2004.04857

\bibitem{GKT2} P. G\'erard, T. Kappeler, P. Topalov, {\em On the spectrum of the Lax operator of the 
Benjamin-Ono equation on the torus},  J. Funct. Anal. $\bf 279$(2020), no. 12, 108762,
arXiv:2006.11864

\bibitem{GKT3} P. G\'erard, T. Kappeler, P. Topalov, {\em On the analytic Birkhoff normal form of the 
Benjamin-Ono equation and applications}, Nonlinear Analysis, $\bf 216$(2022), article 112678

\bibitem{GKT4} P. G\'erard, T. Kappeler, P. Topalov, {\em On the Benjamin-Ono equation on $\T$ and its  
periodic and quasiperiodic solutions}, J. Spectr. Theory, $\bf 12$(2022), no. 1, 169-193

\bibitem{GKT5} P. G\'erard, T. Kappeler, P. Topalov, {\em On the analyticity of the nonlinear Fourier transform 
of the Benjamin-Ono equation on $\T$}, arXiv: 2109.08988


\bibitem{KP-book} T. Kappeler, J. P\"oschel, {\em KdV\& KAM}, $\bf 45$, Springer-Verlag, Berlin, Heidelberg, 2003


\bibitem{KT1} T. Kappeler, P. Topalov, {\em Global well-posedness of KdV in $H^{-1}({\mathbb T},{\mathbb R})$}, 
Duke Journal of Mathematics, $\bf 135$(2006), no. 2,  327--360

\bibitem{KLV} R.~Killip, T.~Laurens, M.~Vi\c san, {\em Sharp wellposedness for the Benjamin--Ono equation}, arXiv: 2304.00124

\bibitem{KS} C. Klein, J.-C. Saut, Nonlinear Dispersive Equations,  Applied Mathematical Sciences, $\bf 209$, AMS, 2021

 

 

\bibitem{Mo} L. Molinet, {\em Global well-posednes in $L^2$ for the periodic Benjamin-Ono equation},
American J. Math, $\bf 130$(3)(2008), 2793-2798


\bibitem{MoP} L. Molinet, D. Pilod, {\em The Cauchy problem for the Benjamin-Ono equation in $L^2$ revisited},
Anal. and PDE, $\bf 5$(2012), no. 2, 365-395

\bibitem{Na1979} A. Nakamura, {\em A direct method of calculating periodic wave solutions to non-linear evolution 
equations. I. Exact two-periodic wave solutions}, J. Phys. Soc. Japan, $\bf 47$(1979), 1701-1705 


\bibitem{PTrub} J. P\"oschel , E. Trubowitz, {\em Inverse Spectral Theory}, Academic Press, 1978

\bibitem{ReedSimon} M. Reed, B. Simon, {\em Methods of Mathematical Physics, II Fourier Analysis, Self-Adjointness}, 
Academic Press, 1975

\bibitem{Saut2019} J.-C. Saut, {\em Benjamin-Ono and intermediate long wave equations: modeling, IST, and PDE},
Fields Institute Communications, $\bf 83$, Springer, 2019


\bibitem{Vey} J. Vey, {\em Sur certains syst\`emes dynamiques s\'eparables}, American Journal  of Mathematics,
$\bf 100$(1978), no. 3, 591-614

\bibitem{Widom} H. Widom, {\em Lectures on Integral Equations}, Van Nostrand Mathematical Studies, $\bf 17$, 1969

\bibitem{Zung} N. Zung, {\em Convergence versus integrability in Birkhoff normal forms}, Ann. of Math., 
$\bf 161$ (2005), 141-156

\end{thebibliography}
\end{document}